\documentclass{amsart}

\usepackage[mathscr]{eucal}
\usepackage{amssymb, amsmath,array, amscd}
\usepackage{url}
\usepackage{graphicx}
\usepackage{lscape}

\input xy
\xyoption{all}

\theoremstyle{change}
\newtheorem{thm}{Theorem}[section]

\newtheorem{prop}{Proposition}[section]

\newtheorem{lemma}{Lemma}[section]

\theoremstyle{definition}

\newtheorem{cor}{Corollary}[section]

\newtheorem{example}{Example}[section]

\def\W{\mathcal W}
\def\F{\mathcal F}
\def\A{\textsf A}
\def\R{\textsf R}
\def\B{\textsf B}
\def\K{\textsf K}
\def\PP{\textsf P}

\def\C{\mathbb C}
\def\P{\mathbb P}
\def\FF{\textsf F}

\def\M{\mathfrak M}

\DeclareMathOperator{\Log}{Log}

\DeclareMathOperator{\Rat}{Rat}

\DeclareMathOperator{\rank}{rank}
\DeclareMathOperator{\Hom}{Hom}
\DeclareMathOperator{\Char}{Char}

\usepackage[active]{srcltx}

\begin{document}
\title{Resonance Webs of Hyperplane Arrangements}
\author{Jorge Vit\'orio Pereira}
\address{IMPA, Estrada
Dona  Castorina, 110\\
22460-320, Rio de Janeiro, RJ, Brazil}
\email{jvp@impa.br}
\date{\today}

\begin{abstract} Each irreducible component of the first resonance variety of a
hyperplane arrangement  naturally determines a codimension one foliation on the ambient space.
The superposition of  these foliations define what we call the resonance  web of the arrangement.
In this paper we initiate the study of these objects  with emphasis on their spaces of abelian relations.
\end{abstract}

\maketitle
\date{}

\section{Introduction}
Let  $\A = \{ H_1, \ldots, H_r\}$ be an arrangement of $r\ge 1$ hyperplanes in $\P^n$.  The complement of $\A$ is an affine variety that will be denoted
by $M=M(\A)$.  It is a result of Arnold \cite{Arnold} ( for the braid arrangement ) and Brieskorn \cite{Brie} ( for and arbitrary hyperplane arrangement )
that the cohomology ring of $M$,  $H^\bullet(M, \mathbb Z)$, is torsion free and generated, as a graded algebra, by the degree one elements
determined by the classes of the logarithmic differential forms
\[
(2 \pi i)^{-1} \left(d \log \frac{h_i}{h_r} \right) \quad  \text{ for } i \in  \{ 1, \ldots, r-1\} \,
\]
where  $h_1, \ldots, h_r$  are  linear polynomials in $\mathbb C[x_0, \ldots, x_n]$ defining the hyperplanes in $\A$.

\smallskip
Given $a \in H^1(M) = H^1(M,\mathbb C)$, consider the complex $(H^{\bullet}(M), a)$ with arrows given by multiplication by $a$:
\[
\xymatrixcolsep{1.2pc} \xymatrix{
0 \ar[r] &H^0(M) \ar[r] &H^1(M) \ar[r] &H^2(M)    \ar[r] &\cdots  \ar[r] &H^n(M) \ar[r] &0\, .
}
\]
The  resonance varieties of $M$, or $\A$,  are defined as
\[
\R^i(M) = \R^i(\A) = \{ a \in H^1(M) ,\ | \, h^i(H^{\bullet}(M),a) \neq 0 \} \, .
\]

The paper \cite{multinets} provides the following nice description of the first resonance variety $\R^1(M)$.  The irreducible components of $\R^1(M)$ are precisely the maximal isotropic subspaces of $H^1(M)$  for the quadratic form
\[
\wedge : H^1(M, \mathbb C) \otimes H^1(M,\mathbb C) \longrightarrow H^2(M,\mathbb C) \,
\]
having dimension at least two.
Moreover, the irreducible components of $\R^1(M)$ of dimension $k$ are in correspondence with pencils of hypersurfaces
on $\P^n$ having exactly $k+1$ elements with support contained in the arrangement.
In particular, for each irreducible component $\Sigma$ of  the resonance variety there is a unique (singular holomorphic) foliation $\mathcal F_{\Sigma}$  on $\P^n$ defined by the corresponding pencil of hypersurfaces.

\smallskip

The interest of the study of the resonance varieties of a hyperplane arrangement is amplified by its relation with the cohomology jumping loci of rank one local systems
on $M$. The  characteristic variety $\Char^i(M)$ of $M$ is the subvariety of $\Hom(\pi_1(M), \mathbb C^*)$ defined as
\[
\Char^i(M) = \{ \rho \in \Hom(\pi_1(M), \mathbb C^*) \, | \, h^i(M, \mathbb C_{\rho} ) \neq 0  \}  \, ,
\]
where $\C_{\rho}$ is the rank one local  system determined by $\rho$.
The above mentioned relation   is given by the following Theorem \cite{Arapura,CohenSuciu}: the exponential map
\begin{align*}
\exp : H^1(M) &\longrightarrow \Hom(\pi_1(M), \mathbb C^*) \\
a &\longmapsto \left( \gamma \mapsto \exp \left(2\pi i \int_{\gamma} a \right) \right)
\end{align*}
defines  isomorphisms between the germs $(\R^i(M),0)$ and   $(\Char^i(M),\mathbf 1)$, where $\mathbf 1 \in \Hom(\pi_1(M), \mathbb C^*)$ is the trivial
representation.
\smallskip

The study of the foliation $\F_{\Sigma}$ for irreducible components $\Sigma \subset \R^1(M)$  led  the author and
S. Yuzvinsky (see \cite{PY}) to   bounds for the dimension of $\Sigma$. Although there
is now (especially after \cite{YS}) a reasonably clear picture about each of the irreducible components of $\R^1(M)$, it is not very clear how
the totality of them sit inside $H^1(M, \C)$. In this paper we propose an approach to produce
invariants  for  arrangements that may turn out to be useful to the study of this question.
The underlying idea is fairly simple: instead of looking at  the foliations associated to
the irreducible components of $\R^1(M)$ one at a time, we should look at all of them at the same time.
More precisely, we will associate to an arrangement $\A$ what we call its resonance web $\W({\A})$ -- the superposition of all the foliations $\F_{\Sigma}$ associate to irreducible components of $\R^1(\A)$ --
and will study its space of abelian relations.

\smallskip

Conversely, many relevant examples for  web geometry, specially in what concerns the dimension of the space of
abelian relations,   appear as  resonance webs of certain arrangements.
The list starts with the Bol exceptional $5$-web \cite{bol36},  contains the Spence-Kummer exceptional $9$-web\cite{Robert,PirioSel},
and ends with some other exceptional webs presented in \cite{Robert,PirioSel}.
This provides further motivation to pursue the study of resonance webs.

\medskip

Our main result  is Theorem \ref{T:1}  which determines  the rank of the
resonance webs for the braid arrangements which, as we will see in Section \ref{S:main}, correspond to
the  $\binom{n+3}{4}$-web on $\M_{0,n+3}$ (the moduli space of $n+3$ distinct ordered points on
$\P^1$) defined by the $\binom{n+3}{4}$ natural maps
\[
\M_{0,n+3} \longrightarrow \M_{0,4} \simeq \mathbb C \setminus \{ 0, 1\} \, .
\]
We will also draw some general considerations about the
abelian relations of resonance webs and use them to  study some of the examples of exceptional webs mentioned above.
Although we have no major results on the structure of the space of abelian relations of resonance webs for arbitrary
arrangements, the blurry picture delineated by   these examples is considerably intricate and, we believe, invites further investigation.

\medskip

\noindent{{\bf Plan of the paper.}} In Section \ref{S:web} we define webs, their spaces of abelian relations, and show how to bound
the rank of arbitrary codimension one webs. We also   review the algebraization results for webs of maximal rank, define exceptional webs, and present Bol's example
of exceptional planar $5$-web. In Section \ref{S:poly} we define  resonance webs and initiate the study of their spaces of abelian relations, more specifically
the subspace of polylogarithmic abelian relations generated by collections of  iterated integrals of logarithmic $1$-forms with poles on the arrangement. The reader will
also find in Section \ref{S:poly}  a brief presentation of a couple of basic results from Chen's theory of iterated integrals relevant to our study.
Section \ref{S:main} is devoted to the statement and proof of our main result: the determination of the rank of the resonance webs of the braid arrangements.
 Section \ref{S:examples} studies some of the exceptional planar webs  found by  Pirio and Robert as  resonance webs  of   line arrangements in $\P^2$.

\smallskip

\section{Web geometry}\label{S:web}

For us a {\bf germ of  codimension one  $k$-web} $\mathcal
W=\mathcal F_1\boxtimes
\cdots \boxtimes \mathcal F_k$ on $(\mathbb
C^n,0)$ is a collection of $k$ germs of smooth codimension one
holomorphic foliations  subjected to the condition that  any two distinct foliations $\F_i, \F_j$ have distinct tangent spaces
at zero.

\smallskip

Usually in the literature a stronger condition is imposed on the tangent spaces at zero. In
the terminology of \cite{CL} the tangent spaces are usually assumed to be in  {\bf strong general position}, meaning that
for any $1\le m\le n$ the intersection of tangent spaces at zero of $m$ distinct foliations $\mathcal F_i$ have
codimension $m$.

Perhaps the most studied invariant of a germ of codimension one web $\mathcal W$ is its {\bf space of abelian relations}
$\mathcal A(\W)$.  If we chose  integrable $1$-forms $\omega_i$ inducing the foliations
$\mathcal F_i$ then $\mathcal A(\mathcal W)$ is equal to
\[
 \left\lbrace
{\big(\eta_i\big)}_{i=1}^k \in (\Omega^1(\mathbb C^n,0))^k \, \,
\Big| \, \,\forall i \, \, d\eta_i =0 \,, \; \eta_i\wedge \omega_i=0
\, \, \text{ and } \sum_{i=1}^k\eta_i = 0 \right\rbrace.
\]

If $u_i:(\mathbb C^n,0) \to (\mathbb C,0)$ are local submersions
defining the foliations $\mathcal F_i$ then, after integration, the
abelian relations can be read as  functional equations of the form $
 \sum_{i=1}^k g_i(u_i) =0
$
for some germs of holomorphic functions $g_i: (\mathbb C,0) \to
(\mathbb C,0)$. Thus we can interpret the abelian relations of $\W$ as functional equations (of a rather special kind)
among the first integrals of the foliations defining it.

\subsection{Rank of webs}

Clearly $\mathcal A(\mathcal W)$ is a vector space and its dimension
is commonly called the {\bf rank} of $\mathcal W$ and is denoted  by
$\rank(\mathcal W)$.  We will now explain how one can bound the rank of arbitrary codimension one
webs. This is a classical subject in web geometry and has been treated by Bol ($n=2$) and Chern ($n\ge 3$ for webs in strong general position)
in the decade of 1930, and more recently by Cavalier-Lehmann for ordinary webs, see  the definition below. Here we will
deal with arbitrary codimension one webs. This section is a summary of  \cite[Section 2.2]{invitation}.

\medskip

For every $i \in \{ 1, \ldots, k\}$ let $\omega_i$ be
 a germ of $1$-form defining $\mathcal F_i$ and satisfying $\omega_i(0) \neq 0 $.
For any positive integer $j$ define $\mathcal L^j(\mathcal W)$ as the subspace of the $\mathbb C$-vector space
$\mathrm{Sym}^j ( \Omega^1(\mathbb C^n,0) )$ generated by
the $j$-th symmetric powers of the exterior forms $\omega_i(0)$ with $i \in \{ 1, \ldots, k\}$.
Set $\ell^j(\mathcal W) = \dim \mathcal L^j(\mathcal W)$.

Alternatively, if  $u_i:(\mathbb C^n,0) \to (\mathbb C,0)$ are local submersions
defining the foliations $\mathcal F_i$, and $h_i$ are their linear terms then
\begin{equation}\label{E:lj}
\ell^j(\mathcal W) = \dim  \left( \mathbb C h_1^j + \cdots + \mathbb C h_k^j \right) \, .
\end{equation}


Notice that the integer  $\ell^j(\mathcal W)$ is bounded by $k$ and by the dimension of the vector space of homogeneous polynomials
of degree $j$ in $n$ variables, i.e.
\begin{equation}\label{E:trivial}
\ell^j ( \mathcal W) \le \max\left\{ k , \binom{n+j -1}{n-1} \right\} \, .
\end{equation}
In the terminology of \cite{CL}, a germ of $k$-web $\mathcal W$ on $(\mathbb C^n,0)$ is {\bf ordinary} if and only if
$\ell^j ( \mathcal W) = \max\left\{ k , \binom{n+j -1}{n-1} \right\}$ for every positive integer $j$.

\medskip

A good lower bound  is harder to obtain. For webs in strong general position there is a lemma
by  Castelnuovo  \cite[Proposition 2.2.2]{invitation} which says that
\begin{equation}\label{E:castelnuovo}
\ell^j ( \mathcal W) \ge \min( k , j(n-1) + 1) ).
\end{equation}
For arbitrary webs, no longer in strong general position, the best possible lower  bound is
$ \ell^j ( \mathcal W) \ge \min( k , j + 1 )$.

\medskip

The  argument used to prove the proposition below is borrowed from
Tr\'{e}preau's proof \cite{Trepreau} of Chern's bound for the rank of webs in strong general position.

\begin{prop}\label{P:bound}
If $\mathcal W$ is an arbitrary   $k$-web on $(\mathbb C^n,0)$  then
\[
\mathrm{rk}(\mathcal  W )\le \sum_{j=1}^{\infty} \max( 0 , k - \ell^j(\mathcal W) )\,
\]
and the sum involves only finitely many non-zero terms.
\end{prop}
\begin{proof}
The space of abelian relations of $\mathcal W$ admits a natural filtration
$\mathcal A(\mathcal W) = \mathcal A^0(\mathcal
W) \supseteq \mathcal A^1(\mathcal W) \supseteq  \ldots \supseteq
\mathcal A^j(\mathcal W) \supseteq  \ldots$,  where
\[
{ \quad
 \mathcal A^j(\mathcal W) = \ker  \left\{ \mathcal A(\mathcal W)\longrightarrow
 \left( \frac{ \Omega^1(\mathbb C^n ,0)}{\mathfrak m^j \cdot \Omega^1(\mathbb C^n
 ,0)} \right)^k \right\}, \quad }
\]
and  $\mathfrak m$ is the maximal ideal of $\mathbb
C\{x_1,\ldots, x_n\}$.

One  can easily verify \cite[Lemma 2.2.6]{invitation} that
\begin{equation}\label{E:partialbounds}
\dim \frac{\mathcal A^j(\mathcal W)}{\mathcal A^{j+1}(\mathcal W)}
\le k - \dim \left( \mathbb C \cdot h_1^{j+1} + \cdots + \mathbb
C \cdot h_k^{j+1}\right) \,
\end{equation}
where $h_i$ is as in (\ref{E:lj}).
The bound  follows. Moreover, as $ \ell^j ( \mathcal W) \ge \min( k , j + 1 ) $ there
are only finitely many non-zero terms at the summation above.
\end{proof}

\smallskip
The proposition above combined with  the lower bounds previously discussed allows us to recover
the bounds for the rank of germs of codimension one webs  available elsewhere.

\begin{cor}\label{C:bound}
Let  $\mathcal W$ be a germ of $k$-web on $(\mathbb C^n,0)$. The  assertions below hold true.
\begin{enumerate}
\item ({\bf Bol's bound}) If $n =2$ then $$\rank(\mathcal W) \le \frac{(k-1) (k-2)}{2}.$$
\item ({\bf Chern's bound}) If $n \ge 3$ and $\mathcal W$ is in strong general position then $$\rank(\mathcal W) \le \sum_{j=1}^{\infty} \max (0 , k - j(n-1) - 1).  $$
\item ({\bf Cavalier-Lehmann's bound}) If $n \ge 3$ and $\mathcal W$ is an ordinary web then
 $$\rank(\mathcal W) \le \sum_{j=1}^{\infty} \max \left(0 , k - \binom{n+j -1}{n-1} \right).  $$
\end{enumerate}
\end{cor}

The number at  right-hand side of Chern's bound is   Castelnuovo's bound  $\pi(n,k)$ for the arithmetic genus
of irreducible, non-degenerated curves of degree $k$ on $\P^n$.

\subsection{Algebraic and algebraizable webs}

An important class of examples of webs is the class of {\bf algebraic webs} which are  webs dual to projective curves.
If $C$ is a reduced  degree $k$  projective curve on $\mathbb
P^n$ then for every general hyperplane  $H_0$   a
germ of codimension one $k$-web $\mathcal W_C$ is naturally
defined on $(\check{\mathbb P^n},H_0)$ through projective duality. More precisely
$\mathcal W_C$ is defined by  the germs of submersions  $p_i : (\check{\mathbb
P^n},H_0) \to C$ characterized by
\[
   H \cdot C = p_1(H) + p_2(H) + \cdots + p_k(H)
\]
for every $H$ sufficiently close to $H_0$.

If $\omega_C$ denotes  the sheaf of abelian differentials on $C$, see for instance \cite[Chapter 3, Section 2.2]{invitation}
for the definition, then Abel's addition Theorem says that for every $p_0 \in C$ and every
regular  $1$-form $\omega \in H^0(C, \omega_C)$ the sum
\[
   \int_{p_0}^{p_1(H)}\omega +    \int_{p_0}^{p_2(H)}\omega + \cdots
   +    \int_{p_0}^{p_k(H)}\omega
\]
does not depend on $H$. One can reformulate this statement as
\[
   \sum_{i=1}^k p_i^* \omega = 0 \, .
\]
It follows that $(p_1^*\omega, \ldots, p_k^*\omega)$ can be interpreted as an  abelian
relation of the algebraic web $\mathcal W_C$. Consequently there is an injection of $H^0(C,\omega_C)$ into $\mathcal A(\mathcal W_C)$.
There is a converse to Abel's addition Theorem ( due to Lie, Poincar\'{e}, Darboux, Wirtinger  see \cite[Chapter 4]{invitation} ) which implies that this
injection is indeed an isomorphism.

Since $h^0(C, \omega_C) = (k-1)(k-2)/2$ for any reduced  plane curve of degree $k$, it follows that every algebraic planar web
has maximal rank. The same is no longer true in higher dimensions as the (arithmetic) genus of curves is not determined by theirs degrees.
The curves giving rise to maximal rank webs are exactly those which have maximal genus in a given degree, the so called Castelnuovo curves.

\subsection{Exceptional webs}\label{S:exceptional}

For $k$-webs in strong general position on $(\mathbb C^n,0)$,  $n\ge 3$,  the maximality of rank
implies that the web is {\bf algebraizable} ( biholomorphic to a web obtained from a projective curve through duality as explained above ) when
 $k\le n+1$ or $k\ge 2n$. This was proved by Bol for $n=3$, and for $n>3$   is a recent
result of Tr\'{e}preau, see \cite{Trepreau}.  The planar case  ($n=2$ ) is rather special in what concerns the classification of webs
of maximal rank. For $k \le 4$ it is well-known that planar webs of maximal rank are algebraizable, the proof
for $k=4$ can be traced back to Lie's work on double translation surfaces, and for $k=3$ is due to  Blaschke-Dubourdieu.
In sharp contrast, \cite{mpp} exhibits examples of non-algebraizable planar    $k$-webs of maximal rank
 for every $k\ge 5$. Further infinite families of examples appear in  \cite{cdql}.  Despite   recent
advances, see for instance \cite{Hen,Hen2,mpp,cdql}, the  classification of {\bf exceptional} ( non-algebraizable and of maximal rank ) planar $k$-webs is wide open.
For a short review of these results see \cite{bourbaki}. A more leisurely account can be found in \cite{invitation}.

\medskip

So far the focus was on   germs of  webs, but  we can   consider webs globally defined
on  a complex variety $X$. For our pourposes it will be sufficient to consider completely decomposable
webs, that is webs $\mathcal W$ which can be globally presented as the superposition of $k$ pairwise distinct global
foliations $\mathcal F_1\boxtimes \cdots \boxtimes \mathcal F_k$.
For the general definition of global webs
 see \cite[Chapter 1]{invitation}.

\smallskip

Given a global web $\mathcal W$, there exists a subvariety $\Lambda = \Lambda(\W) \subset X$ such that
for every $x \in X \setminus \Lambda$ the germ of web $\W_x$ obtained by localizing $\W$ at $x$ is a germ of codimension one
web   and the rank of $\mathcal W_x$ is independent of $x \in X \setminus \Lambda$, see \cite[Theorem 1.2.2]{PTese}.  More precisely, over $X \setminus \Lambda$ the space
of abelian relations of $\W$  is  a local system of $k$-uples of germs of  closed $1$-forms. Therefore it still makes sense to talk
about the rank  for global webs.

\medskip

The first  example of  exceptional web    dates back to the 1936  and was found by Bol, see \cite{bol36}. It is the global $5$-web $\mathcal B_5$
 on $\P^2$ formed by the superposition of $4$ pencils of lines with base points in general position and one pencil of
conics through these four base points. We will explain below that this web is naturally associated to an arrangement of lines on $\P^2$.

\section{Resonance webs}\label{S:poly}

Let  $\A = \{ H_1, \ldots, H_r\}$ be an arrangement of $r\ge 1$ hyperplanes in $\P^n$. Recall from the introduction
that $\R^1(\A)$, the first resonance variety of $\A$,
is the union of the maximal isotropic subspaces of $(H^1(M),\wedge)$ of  dimension at least two.

For $i \in \{ 1, \ldots, r\}$, let $h_i\in \C[x_0,\ldots, x_n]$  be a linear polynomial defining $H_i$.
 From now on we will identify $H^\bullet(M)$ with the algebra generated by the logarithmic $1$-forms $(d \log h_i - d\log h_r)$ with  $i \in \{ 1, \ldots, r-1\}$.

 \smallskip

Before defining our main object of study, let us give a brief idea on how one can associate a pencil of hypersurfaces with $d+1$ completely decomposable fibers
to an irreducible component  $\Sigma$  of  $\R^1(\A)$ of dimension $d$. Let   $\omega_1, \ldots, \omega_d$ be  a basis of $\Sigma$. As $\Sigma$ is
isotropic there exists non-constant rational functions $h_{ij} \in \C(x_0, \ldots, x_n)$ ($i \neq j$) such that  $\omega_i = h_{ij} \omega_j$.
Differentiating this last expression one obtains $0 = dh_{ij} \wedge \omega_j$. Thus the level sets of the functions $h_{ij}$ are tangent
to the distribution determined by $\omega_i$ for every $i \in \{ 1, \ldots, d\}$. Stein factorization theorem ensures the existence of a rational
map $f: \mathbb P^n  \dashrightarrow \mathbb P^1$ such that $df \wedge dh_{ij}= df \wedge \omega_i = 0 $ for  every $i,j \in \{ 1, \ldots, d\}$.
Moreover, there exists $d$ linearly independent logarithmic $1$-forms on $\P^1$, say $\eta_1, \ldots, \eta_d$, such that
$f^*(\eta_i)_{\infty} \subset \A$, and  $f^*(\eta_i)  = \omega_i$ for every $i \in \{ 1, \ldots, d\}$.
The maximality of $\Sigma$ implies that the cardinality of $$\PP_{\Sigma} = \bigcup_{i\in \underline d} (\eta_i)_{\infty}$$ is equal to $d+1$.
Thus the rational map $f= f_{\Sigma}$  determines a pencil of hypersurfaces with $d+1$ fibers contained in the support of the arrangement.
Moreover, the restriction of $f$ to $M = \P^n \setminus \A$ is a regular morphism
\[
 f_{|M}  : M \to C_{\Sigma} \, ,
 \]
where  $C_{\Sigma} = \P^1 \setminus \PP_{\Sigma}$.

\bigskip

Let $\mathcal F_{\Sigma}$ be the foliation on $M$ (or on $\P^n$) determined by the level sets of $f_{\Sigma}$.
We define $\W(\A)$,  the   {\bf resonance web} of $\A$  as the global web on $M$ (or on $\P^n$) obtained by the  superposition of the foliations
$\F_{\Sigma}$ with $\Sigma$ ranging over the irreducible components of $\R^1(\A)$.

\begin{example}
Consider the arrangement on $\P^2$ defined by the polynomial  $\{ xyz(x-z)(y-z)(x-y) =0 \}$. The points of its complement  can be interpreted as isomorphism classes of $5$ ordered points on $\P^1$. To wit,
the  point $(x:y:1) \in \P^2$ satisfying $x - y \neq 0, x\neq 0,1,$ and $y\neq 0,1$ naturally correspond to the $5$-tuple $(0,1,\infty,x,y) \in (\P^1)^5$. We will denote this arrangement by $\A_{0,5}$ and its complement in $\P^2$ by $\M_{0,5}$. The resonance variety of $\M_{0,5}$ has five irreducible components and the associated morphisms are the five forgetting maps
$
\M_{0,5} \longrightarrow \M_{0,4} \, ,
$
sending isomorphism classes of five ordered points on $\mathbb P^1$ to isomorphism classes of four ordered points on $\P^1$. It is a simple exercise
to verify that fibers of four of these maps form  pencils of lines with base points in general position, and that the fibers of  one of them is a pencil of conics
through these base points. Thus Bol's exceptional  $5$-web $\mathcal B_5$ is nothing more than $\W(\A_{0,5})$ the resonance web of $\A_{0,5}$.
\end{example}

\begin{figure}[h]
\begin{center}
\includegraphics[width=4.0cm,height=4.0cm]{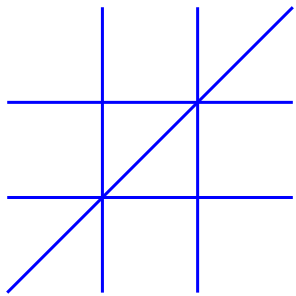}
\caption{Affine trace of the  arrangement $\A_{0,5}$.  }\label{F:bol}
\end{center}
\end{figure}

Due to the important role of Bol's $5$-web in the development of web geometry it is natural to enquire about
the rank of the resonance webs for arbitrary arrangements.
To determine the rank of an  arbitrary web    is a daunting task. The only general method
dates back to Abel ( see \cite{PTese,PirioSel} for a modern account ) and involves lengthy algebraic
manipulations which lead to linear differential equations of high order that have to be solved.
Although implementable, in practice such method is
not computationally efficient and cannot deal with $k$-webs when $k$ is large, say $k> 10$ .

An alternative  approach to compute the rank of certain webs has been devised by Gilles Robert.
Loosely speaking it restricts the search
of abelian relations to a certain class of differential forms  defined from iterated integrals of  logarithmic differentials.
Here  we will explore this approach, adding some topology/combinatorics of arrangements to the picture.

\subsection{Logarithmic abelian relations}\label{S:logar}
Consider the morphism
\begin{eqnarray*}
\Psi_1 : \bigoplus_{I}  H^1(C_{\Sigma}) &\longrightarrow& H^1(M)\\
(\eta_{\Sigma} ) &\mapsto& \sum f_{\Sigma} ^* \eta_\Sigma\, .
\end{eqnarray*}
where $\Sigma$ ranges over all the irreducible components of $\R^1(\A)$. Let
$\Log^1 \W(\A)$ be its kernel, i.e., $\Log^1 \W(\A)= \ker \Psi_1$.

\begin{prop}\label{P:trivial}
The vector space $\Log^1 \W(\A)$ embeds into the space of abelian relations
of $\W(\A)$.
\end{prop}
\begin{proof} Let $(\eta_{\Sigma}) \in \ker \Psi_1$ be a non-zero element.
Each $\eta_{\Sigma}$ corresponds to a logarithmic $1$-form on $\P^1$. The
pull-backs under the corresponding rational map $f_{\Sigma}^* \eta_{\Sigma}$
are  closed logarithmic $1$-forms on $\P^n$, and $f_{\Sigma}^* \eta_{\Sigma}$
defines a distribution tangent to the foliation $\mathcal F_{\Sigma}$. Thus
$(f_{\Sigma}^* \eta_{\Sigma}) \in \A( \W(\A))$ as wanted. \end{proof}

A natural place to look for further abelian relations  is to consider differential forms
with logarithmic coefficients, that is, differential forms like
\[
\log (z) \frac{dz}{z-1}  .
\]
A convenient formalism to deal with such objects is Chen's theory of iterated integrals.

\subsection{Chen's theory of iterated integrals}\label{S:Chen}

In this paragraph $M$ will be an arbitrary connected complex manifold. Given a path $\gamma:[0,1] \to M$  and a collections of $1$-forms $\omega_1, \ldots, \omega_k$ the iterated integral of $\omega_1 \otimes \cdots \otimes \omega_k$
along $\gamma$ is defined as
\[
\int_{\gamma} \omega_1 \otimes \cdots \otimes \omega_k =
\int_{{\Delta_k}} p_1^*\omega_1 \wedge \cdots \wedge p_k^* \omega_k
\]
where $p_i : M^k \to M$ is  the projection on the $i$-th factor and $\Delta_k$ is the image of
standard simplex in $\mathbb R^{k}$ on $M^k$ under the map $\underbrace{\gamma \times \cdots \times \gamma}_{k \, \, \text{times}}$.

\medskip

Even if the $1$-forms $\omega_i$ are closed  this integral {  does} depend on
the path and not just on its homotopy class.
It is a result of Chen \cite[Theorem 4.1.1]{Chen} that the elements of ${\Omega^1(M)}^{\otimes k}$ for which the corresponding iterated integral
does not depend on the representative in a given homotopy class are in the intersection of the kernels of the linear maps, $i \in \{ 1 , \ldots, k-1\}$,
\begin{eqnarray*}
{\Omega^1(M)}^{\otimes k} &\longrightarrow&  \Omega^1(M)^{\otimes i-1} \otimes \Omega^2(M)\otimes \Omega^1(M)^{\otimes (k-i-1)} \\
\omega_1 \otimes \cdots \otimes \omega_k &\mapsto& \omega_1 \otimes \cdots \otimes \omega_{i+1} \wedge \omega_{i+2} \otimes \omega_{i+3} \otimes
\cdots \otimes \omega_k \, ,
\end{eqnarray*}
with the kernels of the  linear maps, $i\in \{1, \ldots, k\}$,
\begin{eqnarray*}
{\Omega^1(M)}^{\otimes k} &\longrightarrow&  \Omega^1(M)^{\otimes i-1} \otimes \Omega^2(M)\otimes \Omega^1(M)^{\otimes (k-i)} \\
\omega_1 \otimes \cdots \otimes \omega_k &\mapsto& \omega_1 \otimes \cdots \otimes d \omega_{i+1}  \otimes \omega_{i+2} \otimes
\cdots \otimes \omega_k \, .
\end{eqnarray*}
If $B^k(M)$ denotes this intersection then every element of $B^k(M)$ gives rise to a
function on the universal covering of $M$ through (iterated) integration. Thus we can interpret
the elements of $B^k(M)$ as closed $1$-forms on the universal covering of $M$ by considering the differential of this function.

\smallskip

Moreover, Chen also proved that  if we consider a vector subspace $V$ of $\Omega^1(M)$  formed by closed $1$-forms with no non-zero exact forms
then the iterated integrals define an injection of  $\oplus_{k \ge 1} V^{\otimes k} \cap B^k(M)$  into the space of holomorphic functions on the
universal covering of $M$. In particular, when $M$ is the complement of a hyperplane arrangement,
this is the case for  $H^1(M)$ seen as a vector subspace of $\Omega^1(M)$. For more about iterated integrals of logarithmic $1$-forms on the complement
of arrangements see  \cite{Kohno}.

\medskip

It is also interesting to observe that $B(M)= \oplus_{k=1}^{\infty} B^k(M)$ admits a natural structure of $\pi_1(M)$-module, with action
defined by analytic continuation of the iterated integrals. Notice that the summands $B^k(M)$  are not $\pi_1(M)$-invariant when $k\ge 2$, but the terms of the filtration $
F^{\bullet} \,  : \, F^k (M) = \bigoplus_{i=1}^k B^i(M)
$
are.

\subsection{Polylogarithmic abelian relations}\label{S:polylogar}
It is natural to extend the construction of Section \ref{S:logar} to
arbitrary iterated integrals of logarithmic $1$-forms. For each $i \ge 1$, consider the morphism
\begin{eqnarray*}
\Psi_i : \bigoplus_{I}  H^1(C_{\Sigma})^{\otimes i} &\longrightarrow& H^1(M)^{\otimes i} \\
(\eta_{\Sigma} ) &\mapsto& \sum f_{\Sigma} ^* \eta_\Sigma\, .
\end{eqnarray*}
where, as before,  $\Sigma$ ranges over all the irreducible components of $\R^1(\A)$. Define
$\Log^i  \W(\A)$ as its kernel, i.e., $\Log^i \W(\A)= \ker \Psi_i$. Define also $\Log^{ \infty}  \W(\A)$ as
the direct sum
\[
\Log^{ \infty}  \W(\A) = \bigoplus_{i=1}^\infty \Log^i \W(\A) \, .
\]

\begin{prop}\label{P:trivial2}
If   $\W$ is  the localization of the  web $\W(\A)$
at a generic point of $M$ then  the vector space $ \Log^{\infty}  \W(\A)$ embeds into the space of abelian relations
of   $\W$. Moreover, the  analytic continuation of this embedding gives rise to a local system of abelian relations
globally defined on $M$.
\end{prop}
\begin{proof}
Fix $i \in \mathbb N$, and let $(\eta_{\Sigma}) \in \ker \Psi_i$ be a non-zero element.
Each $\eta_{\Sigma}$ corresponds to a iterated integral in $C_{\Sigma}$.  The
pull-backs $f_{\Sigma}^* \eta_{\Sigma}$ under the corresponding rational maps 
are  iterated integrals on $M$, and as such can be interpreted as functions on the universal covering of
$M$. Moreover the (multi-valued) functions on $M$ determined by  $f_{\Sigma}^* \eta_{\Sigma}$
are (locally) constant along the leaves of the foliation $\mathcal F_{\Sigma}$.
As we are considering iterated integrals of logarithmic $1$-forms,
the properties of iterated integrals recalled on Section \ref{S:Chen} imply
the functions on the universal covering of $M$ coming from the components of $(\eta_{\Sigma}) \in \ker \Psi_i$
will sum up to zero. Thus the differentials of these functions will define an abelian relation for $\W(\A)$. To check
injectivity one has just to notice that two different elements in $\oplus_i H^1(C_{\Sigma})^{\otimes i}$
will define different functions on the universal covering of $C_{\Sigma}$ according to Chen's Theory,  see Section \ref{S:Chen}. \end{proof}

Let  $k_d( \A )$  be the number of irreducible components of $\R^1(\A)$ of dimension  $d$ and $ k(\A)$ be the total number of
irreducible components of $\R^1(\A)$. Notice that the resonance web of $\A$ is a $k(\A)$-web.

\begin{cor}\label{C:crude}
The following inequalities hold true:
\begin{align*}
\rank(\W(\A)) & \ge  \dim \Log^{\infty} ( \A) \, , \\
 \dim \Log^{\infty} ( \A) & \le   \frac{(k(\A)-1)(k(\A) -2)}{2}\, ,\\
 \dim \Log^i (\A) & \ge  \sum_{d}  d^i k_d (\A) - \dim B^i(M) \cap H^1(M)^{\otimes i}  \, .
\end{align*}
In particular $\Log^{\infty} ( \A) $ is a  finite dimensional vector space.
\end{cor}
\begin{proof} The first inequality follows from Proposition \ref{P:trivial2}. The second
follows from the first  combined with Bol's bound ( Corollary \ref{C:bound} ) for the
rank of planar webs. To prove
the third inequality it suffices to notice that Chen's {\it integrability conditions}
are trivially satisfied by collections of $1$-forms on a curve. Thus the morphism $\Psi_k$
factors as in the diagram below
\[
\displaystyle{\xymatrix{
&  N^k (M) \ar[d] \\
{\bigoplus_{\Sigma}}  H^1(C_{\Sigma} )^{\otimes k}  \ar[r]^{\Psi_k} \ar[ur] &H^1(M)^{\otimes k} \ar[d]\\
&H^2 \otimes {H^1}^{ \otimes k-2} \oplus \cdots \oplus {H^1}^{\otimes k-2} \otimes H^2  \\
}}
\]
where $H^i = H^i(M)$ and  $N^k(M)=B^k(M) \cap H^1(M)^{\otimes k}$.
The corollary follows.
\end{proof}

Since $M$ is the complement of a hyperplane arrangement,  $H^{\bullet}(M)$ is generated in
degree one. Consequently
\[
\dim N^2(M) = h^1(M)^2 - h^2(M) \, .
\]

In general we do not know how to control the dimensions of the vector spaces $N^i(M)$ when $i \ge 3$. Nevertheless
for fiber type arrangements there is the following K\"{u}nneth type formula which is a corollary of  \cite[Theorem 3.38]{Brown}:
If $\mathcal A$ is a fiber type  arrangement on $\mathbb P^n$ with exponents $\{e_1, \ldots, e_n \}$ then
\[
\dim N^i(M) = \sum e_1 ^{j_1} \cdots e_n ^{j_n}
\]
where the sum is over all ordered $n$-uples $0\le j_1 \le \cdots \le j_n$ with $ j_1 + \cdots + j_n = i$.

As will be made clear by the examples in Section \ref{S:examples} the bound for the rank given by   Corollary \ref{C:crude}  is rather crude and does
not capture many otherwise easily  predicable  abelian relations. Nevertheless, we will need not more
than these crude bounds to determine the rank of the resonance webs of the braid arrangements.

\section{Resonance webs of the braid arrangements}\label{S:main}
For $n \ge 2$, let  $\A_{0,n+3}$ be the arrangement of hyperplanes
on $\P^n$ defined by the vanishing of the polynomial
\[
\left( \prod_{i=0}^n x_i \right) \left(\prod_{i=1}^n (x_i -x_0 ) \right) \left(\prod_{1 \le i < j \le n} (x_i - x_j) \right) \, .
\]
It is the quotient of the  braid arrangement $B_{n+2}$ on $\mathbb P^{n+2}$
\[
\prod_{ 0 \le i \le  j  \le n+1 } (y_i - y_j)
\]
by its center $\{y_0= y_1 = \ldots = y_{n+1}\}$. The resonance variety of $\A_{0,n+3}$ is isomorphic to
the resonance variety of $B_{n+2}$, and the resonance web of $B_{n+2}$ is a linear pull-back
of the resonance web of $\A_{0,n+3}$. Consequently, both webs have isomorphic space (local system) of
abelian relations.

The complement of $\A_{0,n+3}$ will be denoted by $\M_{0,n+3}$  and can be identified with the moduli
space of $(n+3)$-uples of pairwise distinct ordered points of $\mathbb P^1$. According to \cite{CohenSuciu}, the resonance variety of $\M_{0,n+3}$ has $\binom{n+3}{4}$
irreducible components which are in correspondence with the forgetful maps  $$\M_{0,n+3} \longrightarrow \M_{0,4}.$$

 Thus the
resonance web of $\W(\A_{0,n+3})$ is a $\binom{n+3}{4}$-web on $\M_{0,n+3} \subset \P^n$. The main result of this paper is the determination of the  rank of $\W(\A_{0,n+3})$  given below.

\begin{thm}\label{T:1}
For every $n \ge 2$ the equality
\[
\mathrm{rank}(\mathcal W ( \A_{0,n+3} ) ) = 3\binom{n+3}{4} -\binom{n+2}{3} - \binom{n+1}{2} - n \, .
\]
holds true.
\end{thm}

The remaining of this Section is devoted to the proof of this Theorem. It will be convenient
to work in the affine chart $x_0=1$.


\subsection{Upper bound for the rank}

For each ordered $4$-uple of ordered integers $1 \le \alpha < \beta < \gamma < \delta \le n +3 $  consider the  map
\begin{align*}
\rho_{\alpha \beta \gamma \delta} : \M_{0,n+3} &\longrightarrow \M_{0,4} \\
(x_1, \ldots, x_{n+3} ) &\longmapsto ( x_{\alpha} , x_{\beta} ,  x_{\gamma}, x_{\delta} )
\end{align*}
where the points in the source and the target represent  isomorphism classes.

Since $\M_{0,4} = \mathbb C - \{ 0, 1\}$ each of these $\binom{n+3}{4}$ maps define isotropic subspaces of
$H^1(\M_{0,n+3})$, namely  $\rho_{\alpha \beta \gamma \delta}^* H^1(\M_{0,4}) \subset H^1(\M_{0,n+3})$.
It can be verified   that these isotropic subspaces are maximal, and that there are no other
resonance varieties for $\A_{0,n+3}$, see  \cite{CohenSuciu}.

\smallskip

Consider now $\mathcal W=\mathcal W(\A_{0,n+3})$, the germification of the resonance web of $\A_{0,n+3}$ at a generic point of
$\M_{0,n+3} \subset \mathbb C^n$. It will be useful to consider the following subwebs:
\begin{enumerate}
\item For each $\alpha \in \underline{n+3}$ let $\mathcal W_\alpha$ be the $\binom{n+2}{3}$-subweb defined by the maps
$\rho_I$ where $I$ ranges over all the ordered $4$-uples containing $\alpha$. Similarly for distinct $\alpha,\beta$ and $\alpha,\beta,\gamma$ in the above range let
$\mathcal W_{\alpha,\beta}$ be the $\binom{n+1}{2}$-subweb and $\mathcal W_{\alpha,\beta,\gamma}$  be the $n$-subweb where $I$ ranges over
the ordered $4$-uples that contains $\alpha,\beta$ and
$\alpha,\beta,\gamma$ respectively.
\item For each $\alpha \in \underline{n+3}$ let $\mathcal W^{\alpha}$ be the pull back of $\mathcal W(\A_{0,n+2})$ under the
morphism $$\rho^{\alpha} : \M_{0,n+3} \longrightarrow \M_{0,n+2}$$ that forgets the  $\alpha$-th point.
\end{enumerate}

We combine  the two constructions above and define
\[
\mathcal W ^{\alpha} _{ \beta, \gamma } = \mathcal W^\alpha \cap \mathcal W_{\beta, \gamma}
\]
where the intersection of the webs is the web  formed by the common foliations of both webs.

\begin{prop}\label{P:boundl}
For each $a \in \{ 1, 2, 3, 4 \}$ and every ordered subset $I$ of $\underline{n+3}$ of cardinality $4-a$, $\mathcal L^a(\mathcal W)$ is
isomorphic to $\mathcal L^a(\mathcal W_I)$. Moreover, the dimension of $\mathcal L^a(\mathcal W)$
is given by the formula
\[
\ell ^a ( \mathcal W ) = \ell ^a ( \mathcal W_I ) = \binom{n+ a-1 } { n-1   } \, .
\]
\end{prop}
\begin{proof}
The proof will be by a double induction on $n$ and $a$.

When $n=2$, $\mathcal W_I= \mathcal W(\A_{0,5})_I$ is an $(a+1)$-web on $\mathbb C^2$. Therefore
$\mathcal L ^a(\mathcal W)$ contains $\mathcal L^a(\mathcal W_I)$, and equation (\ref{E:trivial}) implies
\[
 a+1 \ge \mathcal \ell^a(\mathcal W) \ge \mathcal \ell^a(\mathcal W_I)= a+1.
\]
The result follows in this case.

When $a=1$, we can assume that $I = ( n+ 1, n+ 2, n+ 3)$  and therefore by normalizing the points of $\M_{0,n+3}$ in such way that the last three
are $0,1,\infty$ we have that the $n$ foliations defining $\mathcal W_I$ are defined by the morphisms
$(x_1, \ldots, x_n, 0, 1, \infty) \mapsto (x_i, 0,1,\infty)$. Clearly $\mathcal L^1 ( \mathcal W) = \mathcal L^1 ( \mathcal W_I ) = \oplus_{i=1}^n \mathbb C x_i$.

\smallskip

Suppose now that $a \ge 2$ and $n\ge 3$. Assume  $I \subset \{ n+1, n+2, n+3 \}$, $ j \in \{ 1, \ldots, n \}$, and $(x_{n+1}, x_{n+2}, x_{n+3})= (0,1,\infty)$.
Consider the linear map defined by the derivation  $\frac{\partial }{\partial x_j}$ from $\mathbb C_a[x_1,\ldots,x_n] \simeq Sym^a \Omega^1(\mathbb C^n,0)$
to  $\mathbb C_{a-1}[x_1,\ldots,x_n] \simeq Sym^{a-1} \Omega^1(\mathbb C^n,0)$.
It  induces
the following diagram with exact rows where the unlabeled  arrows are the natural inclusions.
\[
\xymatrix{
0 \ar[r]& \mathcal L^a(\mathcal W_I^j)  \ar[d]\ar[r]  & \mathcal L^a(\mathcal W_I) \ar[d]\ar^{\frac{\partial }{\partial x_j}}[r]  & \mathcal  L^{a-1}(\mathcal W_{I \cup \{ j \} }) \ar[d]\ar[r] & 0 \\
0 \ar[r]& \mathcal L^a(\mathcal W^j)  \ar[r]  &\mathcal L^a(\mathcal W) \ar^{\frac{\partial }{\partial x_j}}[r]  & \mathcal  L^{a-1}(\mathcal W_j) \ar[r] & 0 \\
}
\]
Notice that  $\mathcal W^j$ is isomorphic to $\mathcal W( \A_{0,n+2})$ and therefore, by induction hypothesis,
 the leftmost  vertical arrow is an isomorphism. The induction hypothesis also implies that the vector spaces $\mathcal L^{a-1}(\mathcal W_j)$ and
$\mathcal  L^{a-1}(\mathcal W_{I \cup \{ j \} })$ are both isomorphic to $\mathcal L^{a-1}(\mathcal W)$. Thus the rightmost arrow
is also an isomorphism.
It follows that the middle vertical arrow is also an isomorphism and has dimension given by
\[
\ell ^a ( \mathcal W ) =  \binom{(n-1) + a -1}{ (n-1) -1} + \binom{ n + (a-1) -1}{n-1} = \binom{n+ a-1 } { n-1   } \, .
\]
The proposition follows.
\end{proof}

As a consequence we obtain that for every $n \ge 2$ the webs $\mathcal W(\A_{0,n+3})$ are ordinary webs.

\begin{cor}\label{C:ub}
 For every $n \ge 2$ the inequality
\[
\mathrm{rank}(\mathcal W ( \A_{0,n+3} ) ) \le 3\binom{n+3}{4} -\binom{n+2}{3} - \binom{n+1}{2} - n \, .
\]
holds true.
\end{cor}
\begin{proof}
It suffices to combine Propositions \ref{P:bound} and  \ref{P:boundl}.
\end{proof}

\subsection{Lower bound for the rank}

As there are exactly $\binom{n}{2}+ 2n +1 $ hyperplanes in the (projective)
 arrangement $\A_{0,n+3}$, $h^1(\M_{0,n+3})= \binom{n}{2}+ 2n$. Notice also that  $h^1(\M_{0,4}) =2$.
Since $k(\A_{0,n+3}) = k_2(\A_{0,n+3})$, Corollary \ref{C:crude} implies
\begin{align}\label{E:K1}
\dim\Log^1 \W(\A_{0,n+3}) & \ge 2 \binom{n+3}{4} - h^1(\M_{0,n+3})
\\ &= 2 \binom{n+3}{4} - \binom{n}{2} + 2n.  \nonumber
\end{align}

Similarly,
\[
\dim \Log^2\W(\mathcal A_{0,n+3}) \ge \binom{n+3}{4} \cdot h^1(\M_{0,4})^2  - \dim N^2(\M_{0,n+3}) \, .
\]
It is a result of Arnold \cite{Arnold} that  the Poincar\'{e} polynomial
of $\M_{0,n+3}$ is $P(t) = (1 +2t) (1+3t) \cdots (1 +(n+1)t)$. Therefore the dimension of $N^2(\M_{0,n+3})$ is equal to
\[
h^1(\M_{0,n+3})^2 - h^2(\M_{0,n+3}) = \left(\binom{n}{2} + 2n\right) ^2 - P''(0)/2 \, .
\]

Consequently
\begin{equation}\label{E:K2}
\dim AR_{log}^2(\mathcal A_{0,n+3}) \ge 4 \binom{n+3}{4} -h^1(M_{0,n+3})^2 +  h^2(M_{0,n+3}) \, .
\end{equation}

\subsection{ Proof of Theorem \ref{T:1}}

It is not hard to prove by induction that summing  the right-hand side of the inequalities (\ref{E:K1}) and (\ref{E:K2}) one obtains
$3\binom{n+3}{4} -\binom{n+2}{3} - \binom{n+1}{2} - n$. Therefore, Proposition \ref{P:trivial2} implies that
\begin{eqnarray*}
\mathrm{rank} ( \mathcal W(\mathcal A_{0,n+3} ) )&\ge&  \dim AR_{log}^1(\mathcal A_{0,n+3}) + \dim AR_{log}^2(\mathcal A_{0,n+3}) \\
 &\ge& 3\binom{n+3}{4} -\binom{n+2}{3} - \binom{n+1}{2} - n \, .
\end{eqnarray*}
Combining this lower  bound with the upper bound given in Corollary \ref{C:ub} concludes the proof of  Theorem \ref{T:1}. \qed

\medskip

As the webs $\W (\A_{0,n+3})$ attains Cavalier-Lehmann's bound, it is natural to ask if they are algebraizable or, more generally, linearizable.
As the leaves of  algebraic webs are contained hypersurfaces, every algebraizable web is linearizable but the converse is not always true. In
\cite{PLinear}  it is proved that the webs $\W (\A_{0,n+3})$ are not  linearizable. We refer to this work and references therein for
more about the  linearization of webs.

It is  interesting to compare our Theorem \ref{T:1} with Damiano's determination of the rank of
the (dimension one) web given by the $(n+3)$  maps  \cite{Damiano}
\[
 \M_{0,n+3} \longrightarrow \M_{0,n+2} \, .
\]
These webs turn out to attain the corresponding bound for the rank of one-dimensional webs and are also non-linearizable.

\section{Examples}\label{S:examples}

This section is devoted to the study of
some exceptional planar webs -- first found by Pirio and Robert \cite{PirioSel, Robert}  -- which are resonance webs of suitable line arrangements
in $\P^2$. We use them to  recognize other sources of
abelian relations besides iterated integrals with logarithmic forms with poles on the arrangement.

\subsection{More polylogarithmic abelian relations}

We start with a simple example.
Consider $\A_0$ as the  arrangement of $9$ lines on $\P^2$ with affine trace presented in Figure \ref{F:bad}.
If we suppose that the triple point is at the origin of $\C^2$ then pencil of lines through it corresponds to
an irreducible component of the resonance variety of dimension two. There are other two triple points at the line at infinity and
they also correspond to irreducible components of the  resonance variety of dimension two. There are no other irreducible components
of $\R^1(\A_0)$. The resonance web $\W(\A_0)$ is the  $3$-web determined by the superposition of the foliations given by the level sets
of the functions $x,y,x/y$. It clearly has rank one as
\[
d\log x - d \log y - d\log \frac{x}{y} = 0 \, ,
\]
but $\dim \Log^{\infty} \W(\A_0) = 0$ as one can promptly verify.

\begin{figure}[h]
\begin{center}
\includegraphics[width=4.0cm,height=4.0cm]{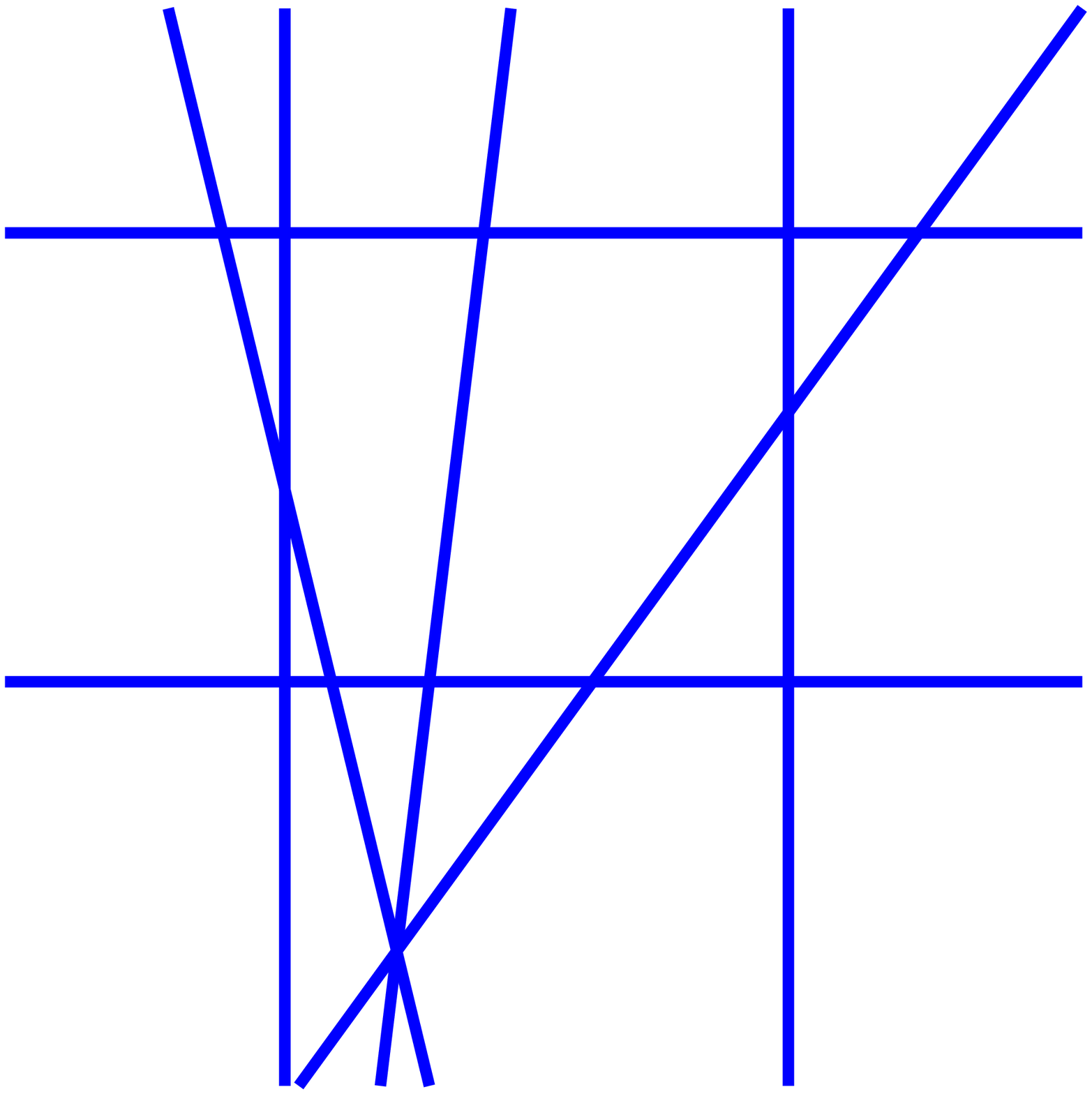}
\caption{ $1= \rank(\W(\A)) > \dim \Log^{\infty} \W(\A) = 0$.    }\label{F:bad}
\end{center}
\end{figure}

\medskip

Given an arrangement of hyperplanes $\A$ on $\P^n$ we define the {\bf resonance closure} of $\A$, denoted by $\overline \A$,
as the arrangement of hypersurfaces on $\P^n$ characterized by the following property: $H \in \overline \A$ if and only
$H \in \A$, or there exists two distinct irreducible components $\Sigma_1, \Sigma_2$ of $\R^1(\A)$ such that
$\dim f_{\Sigma_1}( H) = \dim f_{\Sigma_2}(H) = 0$. In other words either $H$ belongs to the original arrangement $\A$ or
it is invariant by two distinct foliations of the web $\W(\A)$.  The complement of $\overline \A$ will be denoted by $\underline M$.

\begin{example} The resonance closure of the example $\A_0$ is obtained by adding the lines $\{ x=0\}$ and $\{ y =0 \}$. A more interesting
(family of) examples is obtained by considering the (family of) arrangement(s) $\K_5$ determine by the $10$
lines joining $5$ points in $\P^2$  in general position. It can be verified that the resonance variety of $\K_5$ has $10$ irreducible components:
$5$ of dimension $3$ corresponding to pencils the lines through the points, and $5$ of dimension $2$ corresponding to
the pencils of conics through $4$ of the $5$ points.  If $C$ denotes the conic through the $5$ points then
$ \overline{\K_5} = \K_5 \cup \{ C \}$.
\begin{figure}[h]
\begin{minipage}{4.5cm}
\begin{center}
\includegraphics[width=4.0cm,height=4.0cm]{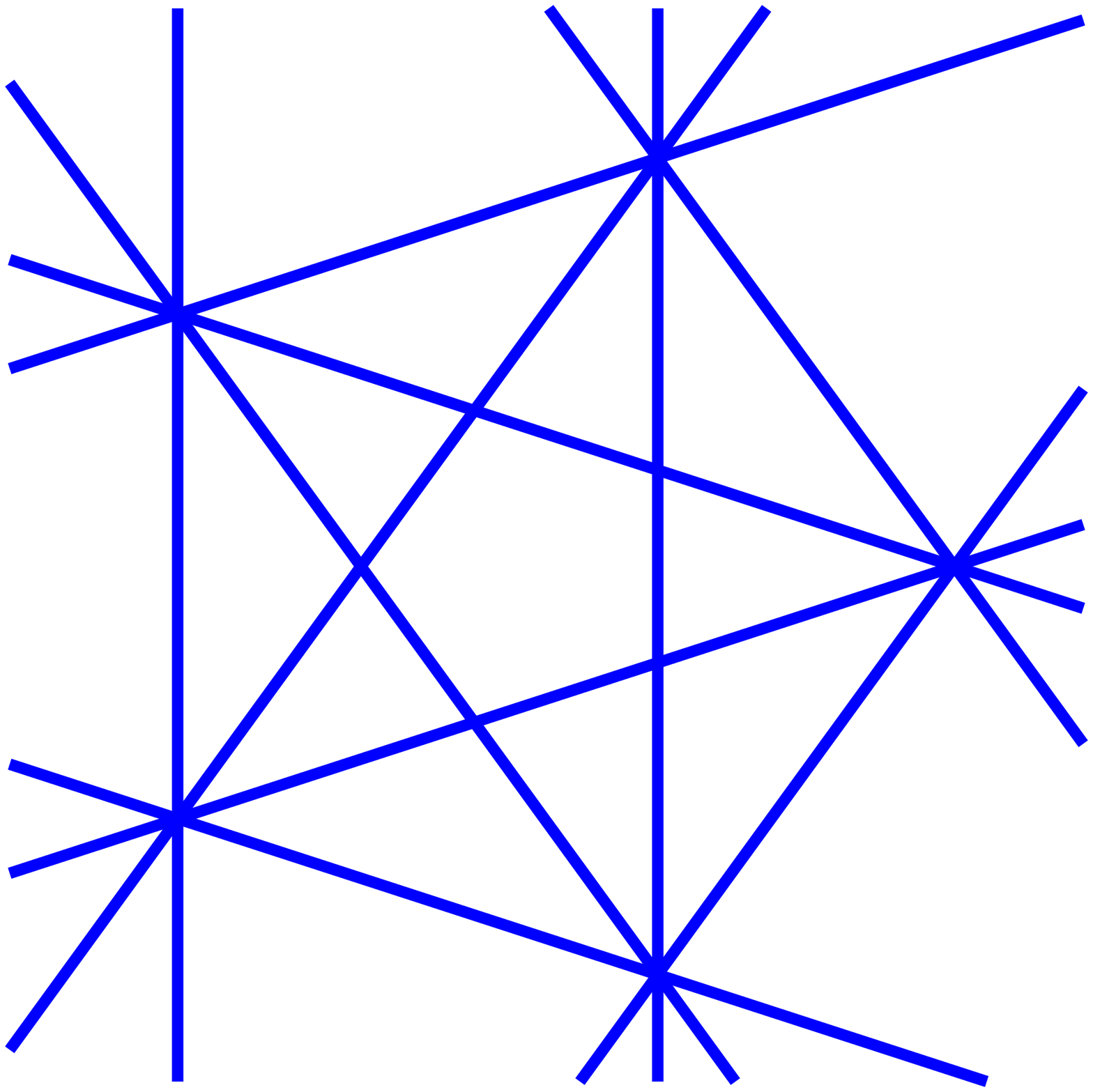}
\end{center}
\end{minipage}
\hfill
\begin{minipage}{4.5cm}
\begin{center}
\includegraphics[width=4.0cm,height=4.0cm]{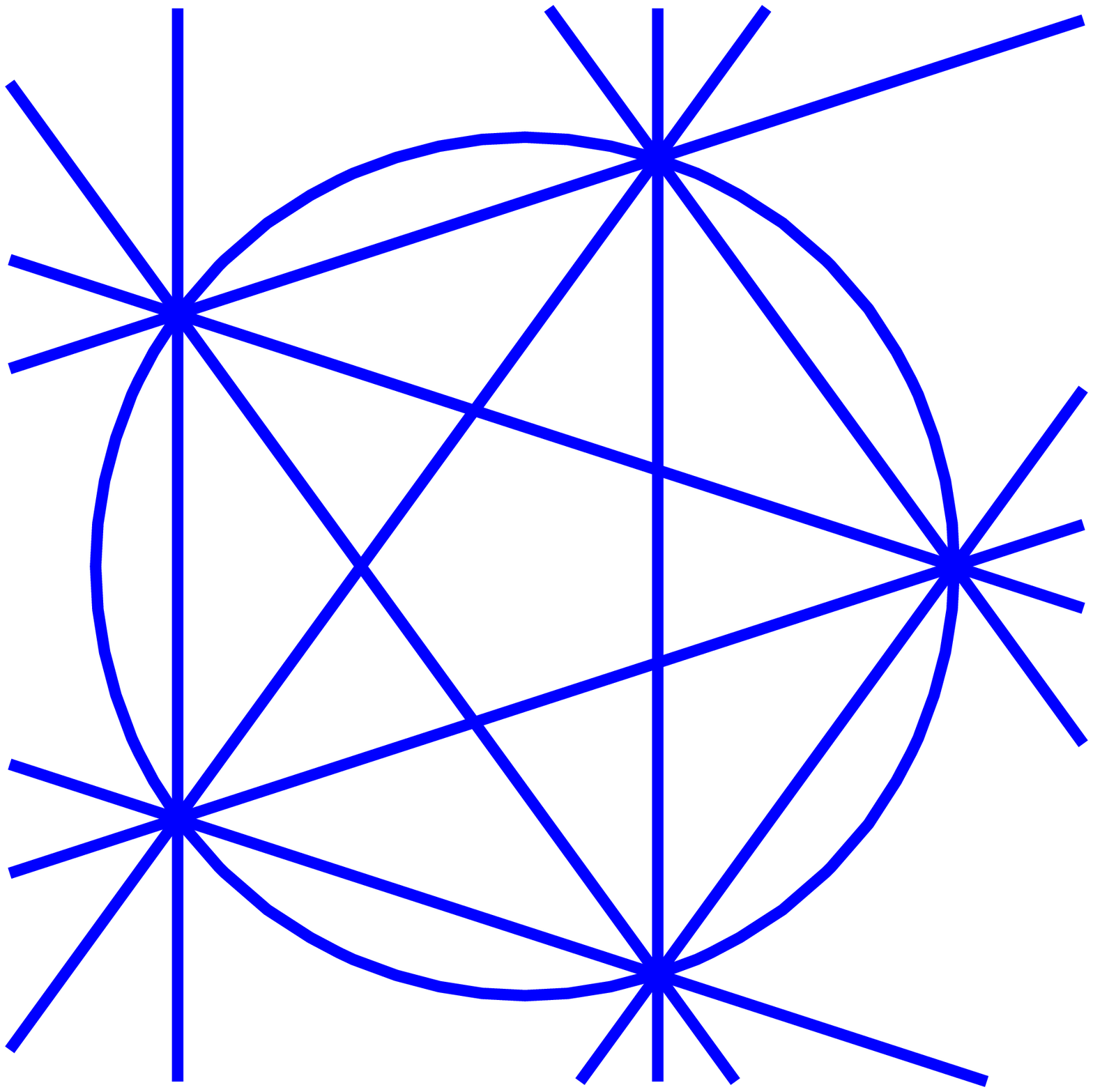}
\end{center}
\end{minipage}
{\caption{On the left  $\K_5$, on the right $\overline{\K_5}$. The line at infinity is not included in the arrangements above.}\label{F:K5}}
\end{figure}
\end{example}

The relevance of the definition of the closure of an arrangement is put in evidence by the
following  specialization of    \cite[Theorem 1.2.2]{PTese}.

\begin{prop}
If $\mathcal W(\A)$ is the resonance web of an arrangement $\A$ then
every germ of abelian relation of $\mathcal W(\A)$ extends to the universal
covering of $\underline M$.
\end{prop}

If $\Sigma$ is an irreducible component of $\R^1(\A)$ we define its closure $\overline \Sigma$ as
the maximal isotropic subspace of $H^1(\underline M)$ containing the image of $\Sigma$ under the
natural inclusion $H^1(M) \to H^1(\underline M)$. The {\bf closure of the resonance variety} of  $\A$
is the subvariety of $H^1(\underline M)$  defined as
\[
\overline{  \R^1(\A)} = \bigcup_{\Sigma \subset \R^1(\A)} \overline{\Sigma} \, .
\]
Notice that  $\overline{\R^1(\A)} \subset \R^1( \overline \A)$ but the equality does not hold in general.

To each irreducible component $\overline{\Sigma} \subset \overline{\R^1(\A)}$ there is a set of points of $\P^1$,  $\PP_{\overline \Sigma}\supset \PP_{\Sigma}$
with complement $\underline{C_{{\Sigma}}}$ satisfying
$
 (f_{\Sigma})^* H^1( \underline{C_{{\Sigma}}}) = \overline{\Sigma}.
$
By analogy with what has been done  in Sections \ref{S:logar} and \ref{S:polylogar}, we can consider
\begin{eqnarray*}
\overline{\Psi}_i : \bigoplus_{\overline{\Sigma}} H^1(\underline{C_{{\Sigma}}})^{\otimes i} &\longrightarrow& H^1(\underline M)^{\otimes i} \\
(\eta_{\Sigma} ) &\mapsto& \sum f_{\Sigma} ^* \eta_\Sigma\, .
\end{eqnarray*}
and define, for every $i \in \mathbb N$,  $\overline{ \Log^i }\W(\A) )  = \ker \overline{\Psi}_i$. We also define $\overline{\Log^{\infty} } \W(\A)$ as the direct sum $\oplus_i\overline{ \Log^i } \W(\A) $.

\smallskip

Of course, we also have a Proposition $\overline{\ref{P:trivial2}}$:

\begin{prop}\label{P:trivial2barra}
If  $\W$ is the localization of the  web $\W(\A)$
at a generic point of $M$ then  the vector space $\overline{ \Log^{\infty} } \W(\A)$ embeds into the space of abelian relations
of   $\W$. Moreover, the  analytic continuation of this embedding gives rise to a local system of abelian relations
globally defined on $\underline{M}$.
\end{prop}

Of course, we can write lower bounds for $\overline{\Log^{i}} \W(\A)$  analogous to the ones given by Corollary \ref{C:crude}.

\smallskip

If $\A= \{ C_1, \ldots, C_r \}$ is an  arrangement of  curves on $\P^2$ it is still a simple matter to determine the dimension of the second cohomology group
of the complement. If $t_i(\A)$  denotes
the number of singular points in the support of $\A$ through which there exactly $i$ branches, and  $\chi(\widetilde{C_i})$ ( $i= 1, \ldots,r$ )
are the Euler characteristic of the normalizations of the curves in the arrangement then  \cite[Proposition 2.4]{Cogo},
\begin{equation}
h^1(M)= r-1 \, \text{  and } \,  \, h^2(M) = 1 + \sum_{i=2}^{\infty}  (i-1)t_i(\A) - \sum_{i=1}^r (  \chi(\widetilde {C_i}) -1  ).
\end{equation}
We can use these observations (together with some computer algebra) to recover the following (unpublished) joint result of Pirio and Robert.

\begin{thm}
The resonance web of $\K_5$ is exceptional.
\end{thm}
\begin{proof}
Notice that the closures of the irreducible components of $\R^1(\K_5)$ all have dimension $3$.
As $h^1(\underline M) = 10$, we obtain the lower bound
\[
\overline{\Log^1} \W(\K_5)  \ge 10 \cdot 3 - h^1(\overline M) = 30 - 10 =20 \, .
\]
To control $\overline{\Log^2} \W(\K_5)$, we need to know the dimensions of  $N^2(\underline M)= \ker \{ H^1(\underline M)^{\otimes 2} \to H^2(\underline M)$.
 For an arbitrary arrangement of curves the cohomology algebra of the complement has no reason to be  generated
in degree $1$, but for $\underline M= \P^2 \setminus \overline{\K_5}$, \cite[Theorem 2.46]{Cogo} implies it is the case.
Thus
\[
\overline{\Log^2} \W(\K_5)  \ge 10 \cdot 9 - h^1(\overline M)^2 + h^2(\overline M) = 90 - 100 + 25 = 15 \, .
\]
The inequalities above turn out to be equalities. To determine the dimension of $\overline{\Log^3} \W(\K_5)$ one has {\it just} to
compute the dimension of the kernel of a $3375 \times 243$ matrix.
A brute-force calculation\begin{footnote}{Maple script  available at \url{www.impa.br/~jvp/artigos.html}.}\end{footnote} shows that
\[
\dim \overline{\Log^3} \W(\K_5) =1 .
\]
Thus $\W(\K_5)$ attains Bol's bound $9\cdot 8 / 2=36$. To prove it is non-algebraizable it suffices to apply \cite[Proposition 2.1]{cdql}.
\end{proof}

Notice that $\K_5$ is indeed a $2$-parameter family of $10$-webs as  the moduli space of isomorphisms classes of  $5$
point on $\P^2$ has dimension $2$.

\subsection{Rational abelian relations}

The inclusion of $\overline{ \Log^{\infty}} \W(\A)$ into $\mathcal A (\W(\A))$ does not exhaust the space of abelian relations of $\W(\A)$ in general, unlike when $\A = \A_{0,5}$ or $\A = \K_5$.
The simplest example is when $\A$ is an arrangement of $9$ lines with $3$ aligned threefold intersection points and all the other intersections are ordinary.
In this case the resonance variety has only three local components, each of them having dimension two.
Supposing that these three points are $(0:1:0), (1:0:0)$, and $(1:1:0)$ then the corresponding foliations are defined by the $1$-forms $dx$, $dy$, and $dx+dy$.
As they satisfy $ (dx) +  (dy) - (dx+dy) = 0$, it is clear that $\W(\A)$ has rank one but $\dim \overline{ \Log^{\infty}} \W(\A)=0$.

Define  $ {\Rat} \W(\A)$ as the kernel of the linear  map
\begin{eqnarray*}
\overline{\Upsilon} : \bigoplus_{{\Sigma}} \frac{\mathbb C(\P^1)}{\C} &\longrightarrow& \frac{\C(\P^n)}{\C} \\
( g_{\Sigma}  ) &\mapsto& \sum f_{\Sigma} ^* (g_\Sigma)\, .
\end{eqnarray*}
where $\C(\P^n)$ stands for the field of rational functions on $\P^n$, and  the summation $\Sigma$ run over all the irreducible
components of $\R^1(\A)$. Coordinate-wise differentiation injects  $\ker \Upsilon$ into $\mathcal A(\W(\A))$. Notice that its  image intersects $\overline{\Log^{\infty}} \W(\A)$ only at zero. Therefore, we have the following lower bound for the rank of $\W(\A)$
\begin{equation}\label{E:lowerrat}
\rank (\W(\A)) \ge \dim \overline{\Log^{\infty}} \W( \A) + \dim  {\Rat} \W(\A) \, .
\end{equation}

We do not know how to give general lower bounds for $\dim  {\Rat} \W(\A)$. We have only the following
simple result.

\begin{lemma}\label{L:ratb}
Let $\A$ be a planar arrangement, and  let $ \{ \ell_1, \ldots, \ell_m\}$ be the lines in its closure. If
$n_i$ is  the number of local irreducible components $\overline{\Sigma} \subset \overline{\R^1}(\A)$ containing
$\ell_i$ in its support then
\[
\dim  {\Rat} \W(\A) \ge \sum_{i=1}^n \frac{ (n_i -1 ) (n_i -2) } {2} \, .
\]
\end{lemma}
\begin{proof}Suppose that $\ell_1$ is the line at infinity. The foliations associated to the   $n_1$ components of $\overline{\R^1}(\A)$
containing $\ell_1$ are defined by $d h_1, \ldots, d h_{n_1}$ where $h_i$ is a linear form. To prove the lemma it suffices to observe that for $p\ge 1$, the kernel of the maps
\[
(a_1, \ldots, a_{n_1}) \mapsto \sum_{i=0}^{n_1} a_i (h_i)^p
\]
will correspond to rational abelian relations with polar set contained in $\ell_1$.
\end{proof}

Of course one can do better by considering rational functions with poles on
fibers of $f_{\Sigma}$ for non-local components of $\R^1(\A)$.

\begin{figure}[h]
\begin{minipage}{3.0cm}
\begin{center}
\includegraphics[width=3.2cm,height=3.2cm]{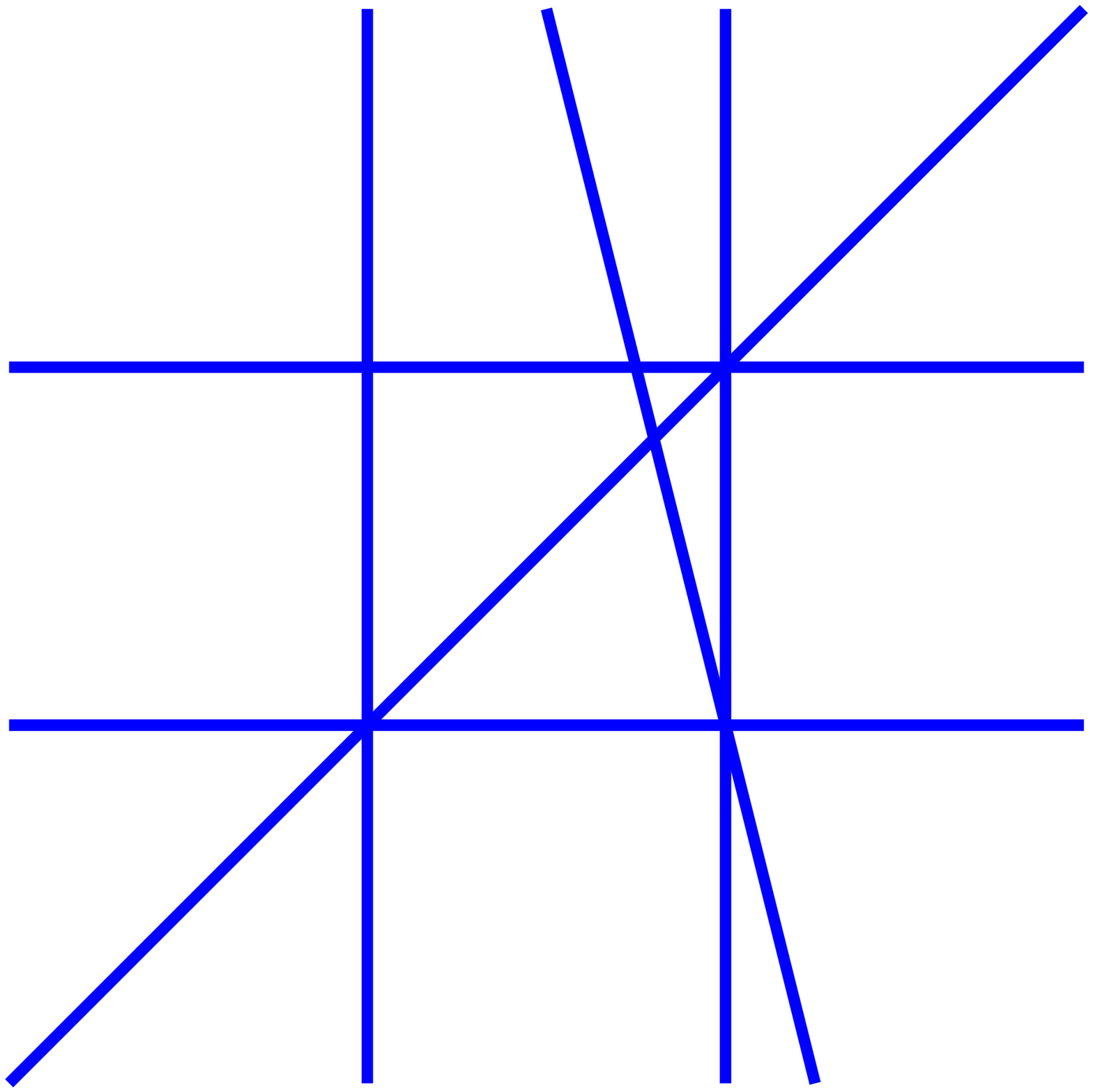}
\end{center}
\end{minipage}
\hfill
\begin{minipage}{3.0cm}
\begin{center}
\includegraphics[width=3.2cm,height=3.2cm]{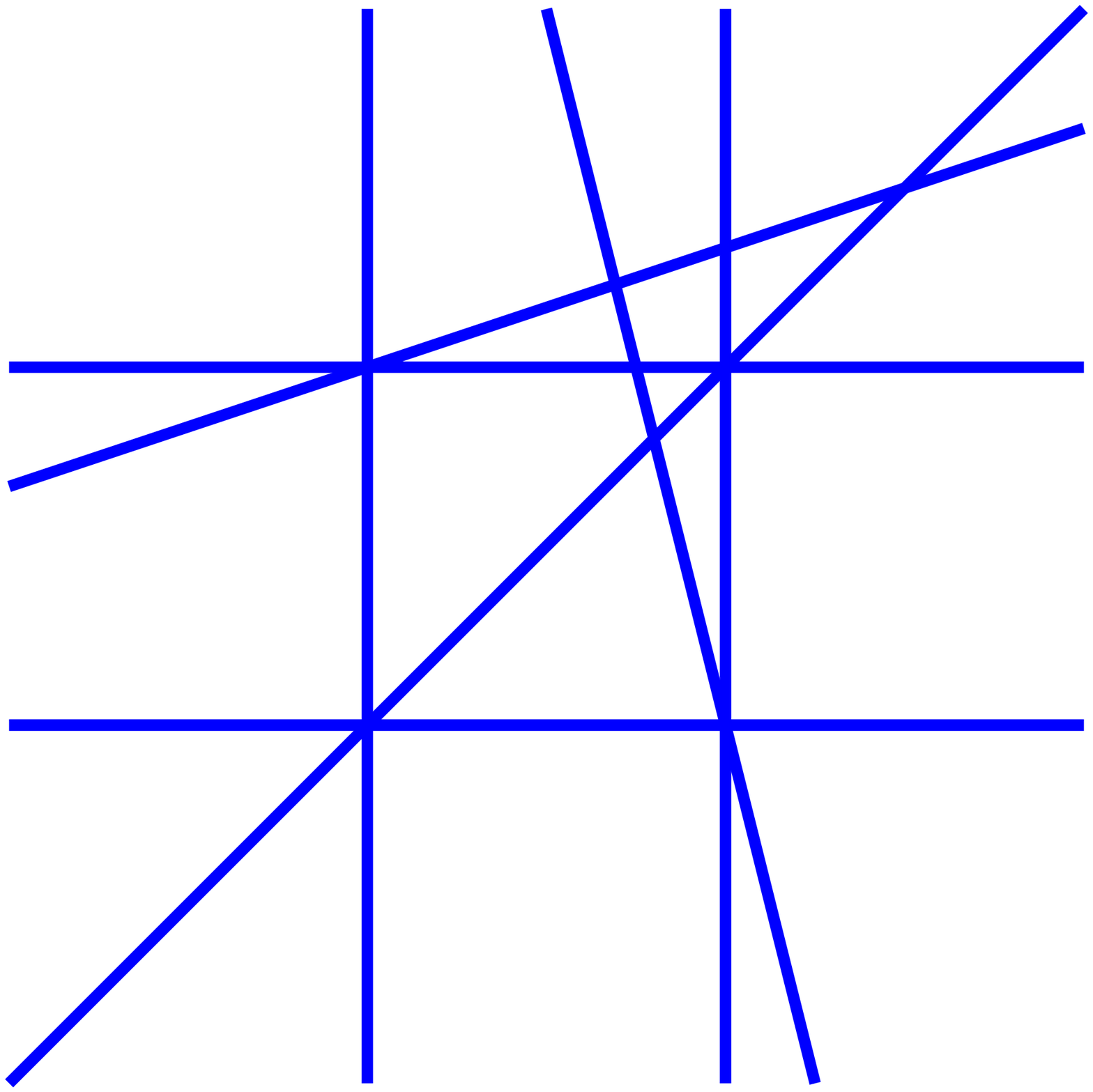}
\end{center}
\end{minipage}
\hfill
\begin{minipage}{3.0cm}
\begin{center}
\includegraphics[width=3.2cm,height=3.2cm]{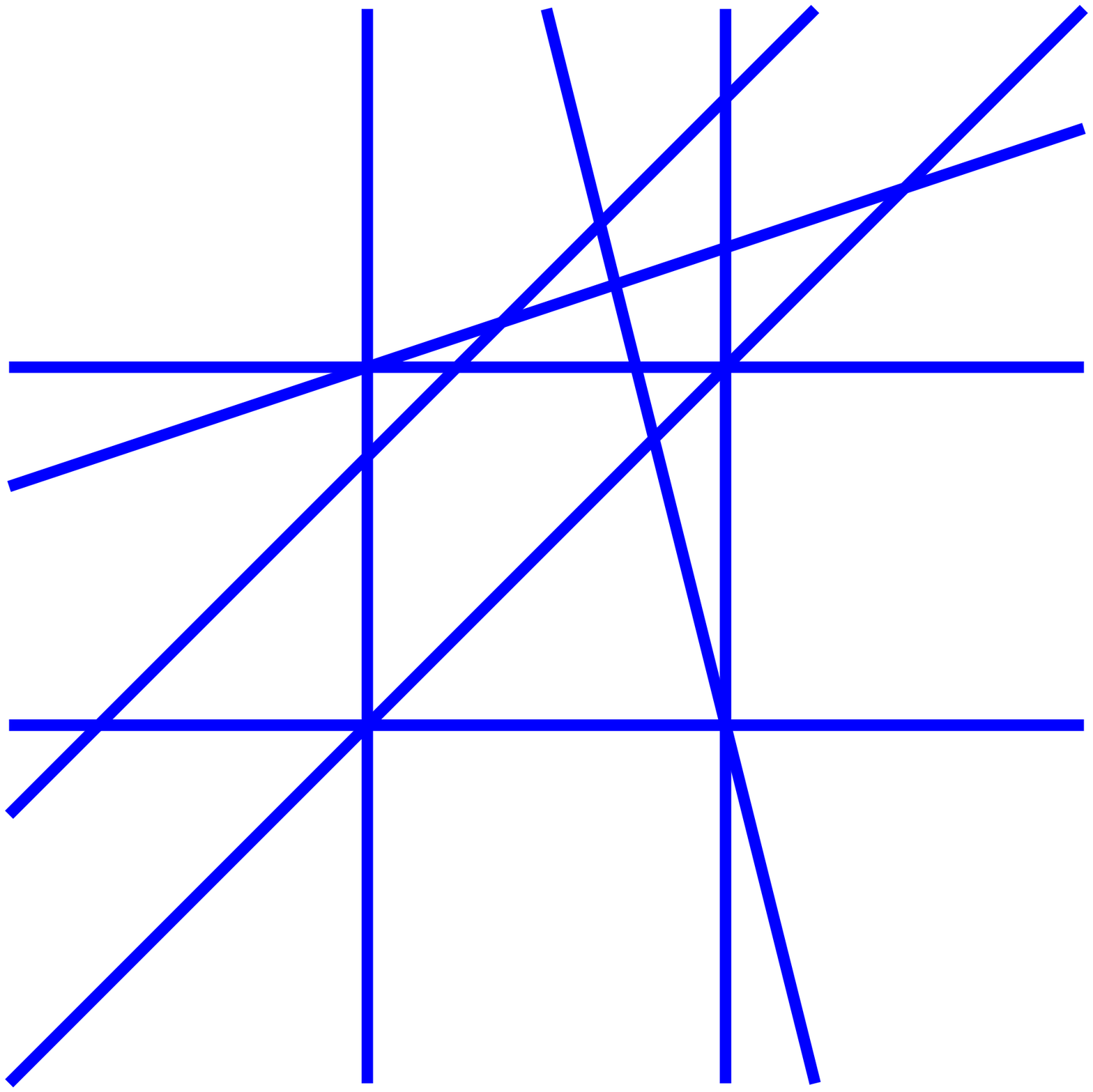}
\end{center}
\end{minipage}
{\caption{Arrangements $\B_6, \B_7$, and $\B_8$. }\label{F:bol678}}
\end{figure}

\begin{example}
Let $\B_{5 + n}$, $n=1,2,3$,  be the arrangements obtained from $\A_{0,5}$ by adding $n$ generic lines, each through a distinct double point, see  Figure \ref{F:bol678}. The resonance webs are the webs $\mathcal B_6$, $\mathcal B_7$ and $\mathcal B_8$ considered by Robert and Pirio, and proved to be exceptional by them.

The rank of $\mathcal B_6$ can be computed easily as follows. As it contains $\A_{0,5}$ it has rank at least $6$. Adding the pencil of lines
through a double point of the support of $\A_{0,5}$ we have that $\overline{\B_6} = \B_6$ contain two lines with three triple points, thus Lemma
\ref{L:ratb} implies $\dim \Rat \W(\B_6) \ge 2$. The proof of Theorem \ref{T:1} tell us that $\Psi_1 : \oplus H^1(C_{\Sigma})\to H^1(M)$
and $\Psi_2 : \oplus H^1(C_{\Sigma})^{\otimes 2} \to H^1(M)^{\otimes 2}$, with $\Sigma \subset \R^1(A_{0,5})$,  are both surjective. If $h_1, h_2$ are the equations of the lines intersecting at the double point under consideration then $d \log h_1 - d \log h_2$ and $(d \log h_1 - d \log h_2)^{\otimes 2}$
belong (respectively) to the image of $\Psi_1$ and $\Psi_2$. Thus
\[
\dim \Log^{ \infty} \W(\B_6) \ge \dim \Log^{\infty} \W(\A_{0,5}) + 2 \, .
\]
Putting all together we deduce that $\rank \W(\B_6) \ge 6 + 2 + 2 = 10$. Thus  $\mathcal B_6 = \W(\B_6)$ is of maximal rank as it  attains Bol's bound.

One can deal similarly with $\B_7$ and $\B_8$ but Lemma \ref{L:ratb} does not suffice. One has to
consider also rational first integrals for the foliation associated to the non-local component of $\R^1(\B_{6})$ with
poles on one (for $\B_7$) or three (for $\B_8$) fibers.
\end{example}

For general arrangements of lines, the inclusion  $$\overline{ \Log^{ \infty} } \W(\A) \oplus { \Rat^{ \infty} } \W(\A)  \subset \mathcal A ( \W(\A))$$
is strict. In the next two sections we will consider two examples of line arrangements on $\P^2$  with resonance webs having  abelian relations
which support this claim.

\subsection{Mixed abelian relations}

In the same way that we looked for abelian relations among collections of iterated integrals of logarithmic $1$-forms we
can also look for them in collections of iterated integrals of arbitrary rational $1$-forms. In this section we will consider
a one-parameter family of arrangements with resonance webs having abelian relations  of this form.

\begin{figure}[h]
\begin{center}
\includegraphics[width=4.0cm,height=4.0cm]{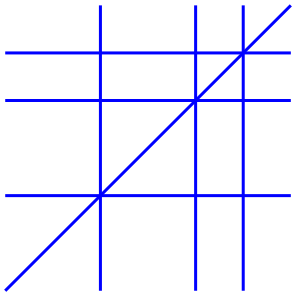}
\end{center}
{\caption{The arrangement $\textsf P$.}\label{F:PP}}
\end{figure}

\begin{example} For $\lambda \in \C \setminus \{ 0, 1\}$ , let $\textsf P = \P_{\lambda}$ be the arrangement of $8$ lines obtained form $\A_{0,5}= \{ x y z ( x-y) (x-z) (y-z)=0\} $ by adding the two extra lines $\{ (x-\lambda z) ( y-\lambda z ) = 0\}$, see Figure \ref{F:PP}. Its resonance variety has $8$ irreducible components: two local of dimension three; three local of dimension two; and three non-local of dimension two determined by pencils of conics. It is closed arrangement in the sense that $\textsf P = \overline{ \textsf P}$.

The family of webs $\W(\PP_{\lambda})$ have been studied in \cite{PirioSel, PTese}. There explicit generators for the corresponding spaces
 of abelian relations are presented.

From Corollary \ref{C:crude} we see that
\[
\dim \Log^1 \W(\PP) \ge  11 \text{ and } \dim \Log^2 \W(\PP) \ge  5 \, .
\]
These inequalities are indeed equalities. Moreover, one can verify that $\dim \Rat \W(\PP)=4$. Thus we have at least a $20$-dimensional
subspace of the space of abelian relations.

There is still one extra abelian relation, relation $G_{21}^{a}$ in Section 3.3 of \cite{PirioSel}, involving the pull-backs under $f_{\Sigma}$ of
\[
 d \left( \frac{\log (x)  }{1-x } \right) = -d \log(x-1)+d \log (x) + {\frac{\log (x)}{(1-x)^2} dx} \, .
\]
The last summand is neither the differential of a  rational function, nor an iterated integral of
logarithmic $1$-forms. Indeed it  can be written as $\frac {dx}{x}\otimes \frac{dx}{(1-x)^2}$, and therefore is an iterated integral of rational $1$-forms.
\end{example}

\subsection{Twisted  logarithmic abelian relations}

Of course the class of  abelian relations with components being iterated integrals of rational $1$-forms encompass all previous classes
of abelian relations considered. Note that if we consider all the abelian relations in this class we obtain an unipotent local system
over the complement of the closure of the arrangement.  We believe that the maximal unipotent local system in $\mathcal A(\W(\A))$
is exactly the one generated by the abelian relations given by iterated integrals of rational $1$-forms.

Our last example shows that the local system $\mathcal A(\W(\A))$ is not in general unipotent.

\begin{example}\label{E:nonfano} Let $\FF$ be the non-Fano arrangement
presented in Figure \ref{F:SK}. It is a closed arrangement ( $ \FF = \overline F$ )  and its resonance variety has $9$ irreducible components: six of them are local of dimension two, and three of them are determined pencil of conics and also have dimension two. The resonance web $\W(\FF)$ is the so called
Spence-Kummer exceptional $9$-web and was
studied independently by Pirio \cite{PirioSel} and Robert \cite{Robert}. They proved that $\W(\FF)$ is an exceptional $9$-web.
The reference to Spence-Kummer  comes from the fact that the foliations of the web
are defined, up to a change coordinates, by the   rational functions appearing in Spence-Kummer functional equation for the trilogarithm
\[
\text{Li}_3(z)  = \int \frac{dz}{z} \otimes \frac{dz}{z} \otimes \frac{dz}{1-z} \, .
\]

\begin{figure}[h]
\begin{center}
\includegraphics[width=4.0cm,height=4.0cm]{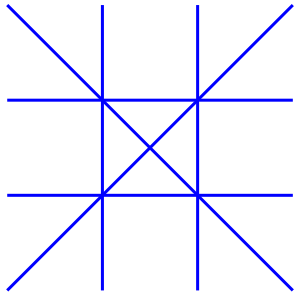}
\end{center}
{\caption{The non-Fano arrangement.}\label{F:SK}}
\end{figure}

Corollary \ref{C:crude} implies
\[
\dim \Log^1 \W(\FF) \ge  12  \text{ and } \dim \Log^2 \W(\FF) \ge  9 \, ,
\]
and Lemma \ref{L:ratb} implies  $\dim \Rat \W(\FF) \ge 4$.
These bounds  are indeed  equalities. Moreover, brute-force computation yields $\dim \Log^3 \W(\FF) =2$.
Thus we have  a $27$-dimensional subspace of the space of abelian relations.

The missing abelian relation comes from the intersection of the irreducible components of
$\Char^1(M)$ determined by the three non-local components. In order to explain this abelian relation we will digress a little.
\end{example}

\begin{center}{\S}\end{center}

Let $\A$ be an arrangement of hypersurfaces on $\P^n$ and $M$ be its complemente. Recall from the introduction the morphism
\begin{align*}
\exp : H^1({ M}) &\longrightarrow \Hom(\pi_1({ M}), \mathbb C^*) \\
a &\longmapsto \left( \gamma \mapsto \exp \left(2\pi i \int_{\gamma} a \right) \right) \,.
\end{align*}
For a $1$-form $\omega \in H^1({ M})$,  let $\varrho_{\omega} : \pi_1({ M}) \to \C^*$ be the representation $\exp(\omega)$ and
$\C_{\omega}$ the corresponding rank one local system.

The $\C$-sheaf $\C_{\omega}$ admits the following resolution
\[
\xymatrixcolsep{1.7pc} \xymatrix{
0 \ar[r] &\C_{\omega}  \ar[r] & \mathcal O_{ M}  \ar[r]^-{\nabla_{\omega}} & \Omega^1({ M})    \ar[r]^{\nabla_{\omega}} & \Omega^2({ M}) \ar[r]^{\nabla_{\omega}} & \cdots \ar[r] &0\,
}
\]
where $\Omega^{\bullet}({ M})$ are the sheaves of holomorphic differentials on ${ M}$ and $\nabla_{\omega} : \Omega^{\bullet}({ M}) \to \Omega^{\bullet +1}({ M})$
is given by the formula $\nabla_{\omega}(\alpha) = d\alpha - \omega \wedge \alpha.$
Since ${ M}$ is Stein, the sheaves $\Omega^{\bullet}({ M})$ are acyclic and consequently
\[
H^i ( { M}, \C_{\omega} ) = \frac{\ker \nabla_{\omega}: H^0({ M}, \Omega^i({ M})) \to H^0({ M},\Omega^{i+1}({ M})) }{ \nabla_{\omega}(H^0({ M},\Omega^{i-1}) )} \, .
\]

If $\omega_1 - \omega_2 = d \log F$ for some $F \in H^0({ M},\mathcal O_{ M})$ then $\mathbb C_{\omega_1} \simeq \mathbb C_{\omega_2}$ and
the corresponding resolutions relate through the diagram
\begin{equation}\label{E:diagrama}
\xymatrixcolsep{1.7pc}
\xymatrix{
0 \ar[r] &\C_{\omega_1}  \ar[r] & \mathcal O_{ M}  \ar[d]^{F^{-1}} \ar[r]^{\nabla_{\omega_1}} & \Omega^1({ M}) \ar[d]^{F^{-1}}   \ar[r]^{\nabla_{\omega_1}} & \Omega^2({ M}) \ar[d]^{F^{-1}} \ar[r]^{\nabla_{\omega_1}} & \cdots  \ar[r] &0\,  \\
0 \ar[r] &\C_{\omega_1}  \ar[r] & \mathcal O_{ M}  \ar[r]^{\nabla_{\omega_2}} & \Omega^1({ M})    \ar[r]^{\nabla_{\omega_2}} & \Omega^2({ M}) \ar[r]^{\nabla_{\omega_2}} & \cdots \ar[r] &0\,
}
\end{equation}
where the vertical arrows are   multiplication by  $F^{-1}$.

For  $\alpha \in \Omega^1(M)$ we have $\nabla_{\omega}(\alpha) = 0$ if and only if the (multi-valued) $1$-form
$
\exp\left( \int \omega \right) \alpha
$
is closed.
Moreover, if $\omega$ belongs to some irreducible component $  \Sigma$  of $ {\R^1}(\A)$ then
for every $\alpha \in   \Sigma$ we have that the $1$-form $\exp\left( \int \omega \right) \alpha$
is closed.

\bigskip

Let $ {\Sigma}$ be an irreducible component of $ {\R^1}(\A)$ and let
$\alpha \in  {\Sigma}$ be a logarithmic $1$-form with residues in $\mathbb Q \setminus \mathbb Z_{>0}$.
Let $P:N \to \P^n $ be the finite abelian covering determined by $\C_{\alpha}$. Since the
monodromy is finite (residues $\in \mathbb Q$),   $N$  is a quasi-projective variety.

\begin{lemma}\label{L:curve}
If $\beta \in  {\Sigma}$ is not a complex multiple of $\alpha$ then $  \exp\left(- \int P^*\alpha \right) P^*\beta$ is a closed
rational $1$-form on $N$ which is not   exact.
\end{lemma}
\begin{proof}
Let $f= f_{ {\Sigma}} : \P^n \dashrightarrow \P^1$ be
the associate rational map. Recall that for $\alpha, \beta \in  {\Sigma}$ there exist logarithmic $1$-forms
$\widetilde{\alpha}$, $\widetilde{\beta}$ on $\P^1$ with poles of $\PP_{\Sigma}$ such that $\beta = f_i^* \widetilde{\beta}$.

If $p:C \to \P^1 \setminus \PP_{  \Sigma}$ is the finite ramified covering determined by $\exp \left( - \int \widetilde{\alpha} )\right)$ then it
fits in the commutative diagram
\[
\xymatrixcolsep{1.7pc}
\xymatrix{
 N \ar[d]_{\widetilde f} \ar[r]^P & M \ar[d]^f \\
 C \ar[r]^-p& \P^1 \setminus \PP_{  \Sigma}
}
\]

Clearly $\exp\left(- \int P^*\alpha \right) P^*\beta$ is equal to $\widetilde{f}^* \exp \left( - \int p^* \widetilde{\alpha} )\right)\widetilde{\beta}$.
Thus if the former is exact, the same holds for the latter.

Suppose $ \exp \left( - \int p^* \widetilde{\alpha} \right)\widetilde{\beta} = d g$ for some rational function on $C$. The rational function $\tilde h=\exp \left(  \int p^* \widetilde{\alpha} \right) g$ is invariant under the covering transformation, thus is equal to $p^*h$ for some rational function on $\P^1$. It is a simple
matter to verify that $\nabla_{\alpha} ( h ) = \beta$,  and therefore $\beta$ represents the zero class in $H^1(\P^1 \setminus \PP_{\Sigma},\C_{\alpha})$.
But the main result of \cite{ESV} implies that the complex $(\Omega^{\bullet}(\P^1 \setminus \PP_{\Sigma}), \nabla_{\alpha})$ is quasi-isomorphhic
to the complex $(H^{\bullet}(\P^1 \setminus \PP_{\sigma}, \C),\wedge \alpha)$. Hence the class of $\beta$   is not zero. \end{proof}

\begin{prop}\label{P:twisted}
Let $\A$ be an arrangement of hyperplanes on $\P^n$ and let $ {\Sigma_1}, \ldots,  {\Sigma_r}$ be
irreducible components of the resonance variety  $  {\R^1}(\A)$. Suppose $\exp( {\Sigma_1}), \ldots, \exp( {\Sigma_r})$ intersect
 at some $\rho \in \Hom(\pi_1( M), \C^*)$
distinct from the trivial representation. If $\overline N$, the projective closure of the finite  covering of $M$ determined by $\rho$,   is simply-connected then
\[
\dim \frac{\mathcal A(\W (\A))}{  {\Log^{\infty}}\W(\A) \oplus   {\Rat}\W(\A)  }\ge - h^1( M, \C_{\omega}) + \sum_{i=1}^r (\dim  {\Sigma_i} -1) \, .
\]
\end{prop}
\begin{proof} Since $\omega$ is in the intersection of $\exp( {\Sigma_1}), \ldots, \exp( {\Sigma_r})$
there exist non-zero logarithmic $1$-forms  $\alpha_i \in  {\Sigma_i}$ ($i= 1, \ldots, r$)  and   rational functions $f_{ij} \in \mathcal O( M)$ (   $i,j = 1, \ldots, r$ ) satisfying
\[
\alpha_i - \alpha_j = d \log f_{ij} \, .
\]
In particular,  $\C_{\alpha_i}$ is isomorphic to $\C_{\alpha_j}$ for every  $i,j$. It is harmless to assume
that all the $1$-forms $\alpha_i$ have non-integer residues and   $\omega = \alpha_1$.

Since $ {\Sigma_i} \subset \ker \nabla_{\alpha_i}$ it follows from (\ref{E:diagrama}) that
$(f_{ij})^{-1} \cdot  { \Sigma_j } \in \ker \nabla_{\alpha_i}$ for every $i,j$.   Therefore
we can define the map
\begin{align*}
\Lambda : \bigoplus_{i=1}^r  \frac{ {\Sigma_i}}{\C \cdot \alpha_i} & \longrightarrow   H^1( M, \C_{\alpha_1} ) \\
( \beta_i)  & \mapsto \beta_1 + \sum_{i=2}^r  (f_{1i})^{-1} \beta_i \, .
\end{align*}
Notice that as we explained before the (multi-valued) $1$-forms
\[
 \exp\left(-\int \alpha_1\right) (f_{1i})^{-1} \beta_i= \exp\left(-\int \alpha_i\right)  \beta_i
\]
are closed.

If $(\beta_1, \ldots, \beta_r)$ belongs to $\ker \Lambda$ then there exists $g \in H^0( M,\mathcal O_{ M})$ for which
\[
\beta_1 + \sum_{i=2}^r (f_{1i})^{-1} \beta_i = dg - g \alpha_1 \,
\]
or, equivalently, for suitable choices of branches of $\exp\left( \int - \alpha_i  \right) $ we have that
\[
 \sum_{i=1}^r \exp\left( \int - \alpha_1  \right) (f_{1i})^{-1}  \beta_i  = d \left(  \exp\left( \int  \alpha_1  \right) g \right)
\]

If we pull-back this equation to $N$ using the finite covering $P:N \to M$ then all the $1$-forms involved
are legitime (univalued) closed rational $1$-forms. Since $\overline N$ is simply-connected, the pull-backs $\widetilde{f_{\Sigma}}$ of the maps
$f_{\Sigma}$ have as target rational curves. Thus the $1$-forms $\exp\left( \int - \alpha_1  \right) (f_{1i})^{-1}  \beta_i $
 can be uniquely written as the pull-back by $\widetilde{f_{\Sigma}}$ of the sum  of an exact rational differential with an logarithmic $1$-form.
Discarding the rational component, one obtains an identity as above but with zero right-hand side. Clearly it is an abelian relation. The linear
independence of the corresponding abelian relations for $(\beta_1, \ldots, \beta_r)$ varying on a basis of $\ker \Lambda$ follows
from Lemma \ref{L:curve}.
\end{proof}

We do not know if the hypothesis made on the topology of $\overline N$ is  necessary to prove the proposition above.

\begin{center}{\S}\end{center}

Back to the analysis of Example \ref{E:nonfano}, we have that the exponential of the three non-local components intersect at a representation
$\rho$ for which $h^1(M, \C_{\rho})=2$. Moreover, $\rho$  satisfies the hypothesis of  Proposition \ref{P:twisted} as it is non-trivial
only along two fibers of the corresponding rational maps $f_{\Sigma}$, see \cite[Example 10.5]{CohenSuciu}.
Thus  Proposition \ref{P:twisted}   ensures  the existence of the {\it missing} abelian relation of $\W(\F)$.

\subsection{Final remarks}

In the table below we present the dimensions of the subspaces of the
space of abelian relations of resonance webs of line arrangements studied in this paper.

\medskip
\renewcommand\arraystretch{1.4}
\begin{center}{
{\footnotesize{
\begin{tabular}
{|c|c|c |c|c|c|c|c|c|c|c|}
  \hline
  &$k$&$  \Log^{1}$&$\overline{\Log^{1}}$&$  \Log^{2}$&$\overline{\Log^{2}}$&$\Log^{3}$&$\overline{\Log^{3}}$&$\Rat$  & $\text{Mixed}$& $\text{Twisted}$\\
  \hline \hline
  $\A_{0,5}$ &$5$ &$5$    & $5$  & $1$  & $1$  & $0$  & $0$ & $0$& $0$ & $0$\\
  \hline
  $\B_6$&$6$ &$6$    & $6$  & $2$  & $2$  & $0$  & $0$ & $2$& $0$ & $0$\\
  \hline
    $\B_7$&$7$ &$7$    & $7$  & $3$  & $3$  & $0$  & $0$ & $5$& $0$ & $0$\\
    \hline
    $\B_8$&$8$ &$8$    & $8$  & $4$  & $4$  & $0$  & $0$ & $9$& $0$ & $0$\\
    \hline
    $\PP$&$8$&$11$    & $11$  & $5$  & $5$  & $0$  & $0$ & $4$& $1$ & $0$\\
        \hline
    $\FF$&$9$&$12$    & $12$  & $9$  & $9$  & $2$  & $2$ & $4$& $0$ & $1$\\
        \hline
  $\K_5$&$10$&$16$  & $20$ & $5$ &$15$ & $0$ & $1$ & $0$ & $0$ & $0$  \\
\hline
\end{tabular}
}}}\end{center}

\medskip

Although we have studied resonance webs for hyperplane arrangements one can
study resonance webs for arbitrary hypersurfaces arrangements on $\P^n$.
Even more generally, if   the cohomology algebra  $H^{\bullet}(M)$ is replaced
by a finite dimensional algebra of differential forms on a quasi-projective variety then one can
still talk about resonance varieties. Its irreducible components are still
in correspondence with rational maps to projective curves \cite{Bauer}, and
consequently one can still define the resonance webs. We are not aware of any
 exceptional web  arising this way that are not
pull-backs by rational maps of  resonance webs of  hyperplane
arrangements.

\medskip

\noindent{{\bf Acknowledgements.}} This paper would not come to light without the warm reception given to
my talk at the 2nd MSJ-SI Arrangements of Hyperplanes. I would like to thank the organizers
for the opportunity to participate in this  meeting.
I am also indebt to Luc Pirio who,  through numerous discussions, contributed a lot to this paper.
Among other things, he  suggested that the resonance webs of $\A_{0,n+3}$
should be looked at,  gave  a computer-assisted proof of   Proposition \ref{P:boundl}   in the case $n=3$,
and brought  reference \cite{Brown}   to my knowledge.

\end{document}